\documentclass[reqno, 11pt]{amsart}
\usepackage{amssymb}
\usepackage[hypertex]{hyperref}

\usepackage{verbatim}

\newtheorem{theorem}{Theorem}[section]
\newtheorem{corollary}[theorem]{Corollary}
\newtheorem{lemma}[theorem]{Lemma}
\newtheorem{proposition}[theorem]{Proposition}

\theoremstyle{definition}
\newtheorem{remark}[theorem]{Remark}

\theoremstyle{definition}

\theoremstyle{definition}
\newtheorem{assumption}[theorem]{Assumption}

\makeatletter
\def\dashint{\operatorname%
{\,\,\text{\bf--}\kern-.98em\DOTSI\intop\ilimits@\!\!}}
\makeatother

\def\.5{\frac{1}{2}}

\def\bR{\mathbb{R}}

\def\bH{\mathbb{H}}

\def\cD{\mathcal{D}}

\def\h{\mathcal{H}}

\def\o{\mathcal{O}}

\def\l{\mathcal{L}}

\begin{document}
\title[Conormal problem]
{Conormal problem of higher-order parabolic systems}

\author[H. Dong]{Hongjie Dong}
\address[H. Dong]{Division of Applied Mathematics, Brown University,
182 George Street, Providence, RI 02912, USA}
\email{Hongjie\_Dong@brown.edu}
\thanks{H. Dong was partially supported by the NSF under agreement DMS-1056737.}

\author[H. Zhang]{Hong Zhang}
\address[H. Zhang]{Division of Applied Mathematics, Brown University,
182 George Street, Providence, RI 02912, USA}
\email{Hong\_Zhang@brown.edu}
\thanks{H. Zhang was partially supported by the NSF under agreement DMS-1056737.}

\subjclass[2010]{35J58, 35R05, 35B45}

\keywords{higher-order
systems, $L_p$ estimates, Schauder estimates}

\begin{abstract}
The paper is a comprehensive study of the $L_p$ and the Schauder estimates for higher-order divergence type parabolic systems with discontinuous coefficients in the half space and cylindrical domains with conormal derivative boundary condition.
For the $L_p$ estimates, we assume that the leading coefficients are only bounded measurable in the $t$ variable and $VMO$ with respect to $x$. We also prove the Schauder estimates in two situations: the coefficients are  H\"older continuous only in the $x$ variable; the coefficients are H\"older continuous in both variables.
\end{abstract}

\maketitle

\setcounter{tocdepth}{1} \tableofcontents

\section{Introduction}
This paper is devoted to the study of the $L_p$ and Schauder estimates for higher-order divergence type parabolic systems with conormal derivative boundary condition.

Many authors have studied the $L_p$ estimates {for} parabolic systems with discontinuous coefficients.  It is of particular interest not only because of its various important applications in nonlinear equations, but also due to its subtle link with the theory of stochastic processes. A good reference is \cite{Kry04}.

In many references including, for example, \cite{ADN64}, the $L_p$ estimates for systems with constant or continuous coefficients are obtained by relying on the exact representation of solutions and the Calder\'{o}n--Zygmund theorem. Another approach for such $L_p$ estimates is that of Campanato--Stampachia using Stampachia's interpolation theorem (see \cite{Gia93}). In this paper, we expand the $L_p$ theory of higher-order parabolic systems to include a class of discontinuous coefficients.

For systems in the whole space and in the half space with the Dirichlet boundary condition, the first author and Kim \cite{DK10} obtained the $L_p$ estimates with coefficients VMO in the spatial variable $x$ (also denoted by $\text{VMO}_x$), under the Legendre--Hadamard ellipticity condition:
\begin{equation}
 \sum_{|\alpha|=|\beta|=m}A^{\alpha\beta}_{ij} \xi_i\xi_j\ \eta^\alpha\eta^\beta\ge \delta |\xi|^2|\eta|^{2m},\label{hadamard}
\end{equation}
where  $m$ is a positive integer,  $A^{\alpha\beta}$ are $n\times n$ matrices, $\xi \in \mathbb{R}^n$, $\eta\in \mathbb{R}^d$,  $\delta>0$ is a constant, $ \alpha=(\alpha_1,\cdots,\alpha_d)$ and $\eta^\alpha=\eta_1^{\alpha_1}\eta_2^{\alpha_2}\cdots\eta_d^{\alpha_d}$. The current paper can be viewed as a continuation of \cite{DK10}. To present our results, we let
\begin{equation*}
\l u=\sum_{|\alpha|\le m, |\beta|\le m}D^{\alpha}(A^{\alpha\beta}D^\beta u),
\end{equation*}
where $D^\alpha=D_1^{\alpha_1}\cdots D_d^{\alpha_d}$.
The functions used throughout this paper
\begin{equation*}
u=(u^1,\cdots, u^n)^{tr},\quad f_\alpha=(f_\alpha^1,\cdots, f_\alpha^n)^{tr}
\end{equation*}
are real vector-valued functions.
For the $L_p$ estimates, the parabolic systems we consider are of the form
\begin{equation}
u_t+(-1)^m \l u+\lambda u=\sum_{|\alpha|\le m}D^\alpha f_\alpha \label{the system}\quad \text{in}\quad (-\infty, T)\times \Omega
\end{equation}
with the conormal derivative boundary condition,
where $\lambda\ge 0, \Omega\subset\bR^{d}$ and $\Omega$ is either the half space or a $C^{m-1,1}$ domain.
 That is,
\begin{align*}
&\int_{(-\infty,T)\times \Omega}( -u\phi_t+\sum_{|\alpha|\le m, |\beta|\le m}(A^{\alpha\beta}D^\beta uD^\alpha \phi)+\lambda u\phi)\,dx\,dt\\
&=\sum_{|\alpha|\le m}(-1)^{|\alpha|}\int_{(-\infty,T)\times \Omega} f_\alpha D^\alpha \phi \,dx\,dt
\end{align*}
for any $\phi \in C_0^\infty((-\infty,T)\times \overline{\Omega})$.

We prove that if the coefficients $A^{\alpha\beta}(t,x)$ are VMO in the spatial variable $x$ and satisfy the strong ellipticity condition (for a precise definition see \eqref{strong}), then the solution $u$ of \eqref{the system} with the conormal derivative boundary condition satisfies the following estimate
\begin{equation*}
\sum_{k=0}^m\lambda^{1-\frac{k}{2m}}\|D^ku\|_{L_p((-\infty, T)\times \Omega)}\le N\sum_{|\alpha|\le m}\lambda^{\frac{|\alpha|}{2m}}\|f_\alpha\|_{L_p((-\infty, T)\times \Omega)},
\end{equation*}
where $N$ is a constant independent of $u$ and $\lambda$ is sufficiently large. We note that in the second-order case such result has been proved in \cite{DK11b} by using the result in the whole space and the technique of odd/even extensions. However, such technique does not work for higher-order equations or systems.

Our proof is in the spirit of an approach introduced by Krylov \cite{Kry07,Kry071} to deal with the second-order elliptic and parabolic equations with $\text{VMO}_x$ coefficients in the whole space, which is well explained in his book \cite{Kry08}. Generally speaking, this approach consists of two steps. First, we establish mean oscillation  estimates for systems with simple coefficients, i.e., the coefficients depending only on $t$. Second, we use a perturbation argument  which is well suited to the mean oscillation  estimates together with the Fefferman--Stein theorem on sharp functions and the Hardy--Littlewood theorem on maximal functions to obtain the desired $L_p$ estimates.

In our case, we need to distinguish $D_d^m u$ from the other $m$th derivatives of $u$. Here are the new ingredients of our proof. we begin by considering special systems with simple coefficients as follows.  Let $u$ be a solution of
\begin{equation}
u_t+(-1)^m \sum_{j=1}^{d-1}D_j^{2m} u+(-1)^mD^{\hat{\alpha}}(A^{\hat{\alpha}\hat{\alpha}}(t)D^{\hat{\alpha}}u)=0 \quad \text{in} \quad \{x_d>0\}\label{eq1.2}
\end{equation}
with the conormal derivative boundary condition on $\{x_d=0\}$, where $\hat{\alpha}=(0,\ldots, {0,m})$.
One can show that the conormal boundary condition corresponding to the special system is given by
\begin{equation*}
D_d^m u=\cdots=D_d^{2m-1}u=0\quad \text{on}\quad \{x_d=0\}.
\end{equation*}
This allows us to use some estimates for non-divergence type systems with the Dirichlet boundary condition obtained in \cite{DK10}. More precisely, we differentiate \eqref{eq1.2} $m$ times with respect to $x_d$, and observe that $D_d^m u$ satisfies the Dirichlet boundary condition on $\{x_d=0\}$. We then apply a result in \cite{DK10} to obtain a H\"older estimate and consequently a mean oscillation estimate of $D_d^m u$.  The latter together with the Fefferman--Stein theorem and the Hardy--Littlewood theorem yields the $L_p$ estimate of $D_d^m u$ in the half space for the special systems. For general systems in the half space with simple coefficients, we implement a scaling argument to control the $L_p$ norm of $D_d^m u$ by the $L_p$ norms of $f_\alpha$ and $D_{x^\prime}D^{m-1}u$, which in view of the Sobolev embedding theorem implies a H\"older estimate of $D^{m-1}u$. Now using the fact that we can always differentiate the system with respect to $x^\prime$, we are able to prove a local H\"older estimate of $D_{x^\prime}D^{m-1}u$ for homogeneous systems. A mean oscillation estimate of $D_{x^\prime}D^{m-1}u$ then follows from the H\"older estimate naturally for inhomogeneous systems with simple coefficients. For general coefficients, we use the argument of freezing the coefficients to obtain a mean oscillation estimate for $D_{x^\prime}D^{m-1}u$ and a scaling argument  to bound the $L_p$ norm of $D_d^m u$ by $f_\alpha$ and $D_{x^\prime}D^{m-1}u$ just as in the case of the simple coefficients. Combining the estimates of $D_{x'}D^{m-1}u$ and $D_d^m u$, we prove the $L_p$ estimates.

In \cite{FMP91, FMP93}, Chiarenza, Frasca, and Longo initiated the study of the $W_p^2$ estimates for second-order elliptic equations with VMO leading coefficients. Their proof is based on certain estimates of the Calder\'{o}n--Zygmund theorem and the Coifman--Rochberg--Weiss commutator theorem. Recently, the first author and Kim \cite{DK12} considered the conormal problem for higher-order elliptic systems with coefficients merely measurable in one direction and have small mean oscillation in orthogonal directions on each small ball. Regarding other developments in this direction, we refer the reader to  Bramanti and Cerutti \cite{BC93}, Bramanti, Cerutti, and  Manfredini \cite{BCM96},
Di Fazio \cite{F96}, Maugeri, Palagachev, and Softova \cite{MPS}, Palagachev and Softova \cite{PS}, Krylov \cite{Kry08}, and the references therein.

The second objective of this paper is to obtain the Schauder estimates for solutions to the following system:
\begin{equation}\label{the system1}
u_t+(-1)^m \l u=\sum_{|\alpha|\le m}D^\alpha f_\alpha \quad \text{in}\quad (0,T)\times \Omega
\end{equation}
with the conormal derivative boundary condition on $(0,T)\times \partial \Omega$ and the zero initial condition on $ \{0\}\times \Omega$, where $\partial\Omega\in C^{m,a}$, for some $a\in (0,1)$.
There is a vast  literature on the Schauder estimates for parabolic and elliptic equations; see, for instant, \cite{FA, L96, Kry96}. The classical approaches are based on analyzing the fundamental solutions of the equations with a perturbation argument.  On the other hand, when dealing with systems, it has become customary to use Campanato's technique which was first introduced in \cite{Cam66}, and is well explained in \cite{Gia93}.
However, most of these results are obtained under the assumption that the coefficients are sufficiently regular in both $t$ and $x$.

In this paper, for systems in the half space we estimate all the $m$th order spatial derivatives with the exception of $D_d^{m}u$, when the coefficients are only measurable in the $t$ variable. This type of coefficient has been studied by several authors mostly for second-order equations; see, for instance, \cite{Knerr,L92, Lorenzi, KrPr}. Lieberman \cite{L92} studied the interior and boundary Schauder estimates for second-order parabolic equations with the same class of coefficients. In the proof, he used Campanato type approach and the maximum principle, the latter of which no longer works for systems or higher-order equations.  Here we implement the $L_p$ estimates obtained in the first part, with a bootstrap argument to obtain a local H\"{o}lder regularity for systems with simple coefficients, which yields the following mean oscillation estimate:
\begin{align*}
&\int_{Q_{r}^+(X_0)}|D_{x^\prime}D^{m-1}u-(D_{x^\prime}D^{m-1}u)_{Q_r^+(X_0)}|^2\,dx\,dt\nonumber\\
&\le N(\frac{r}{R})^{2\gamma+2m+d}\int_{Q_{R}^+(X_0)}|D_{x^\prime}D^{m-1}u-(D_{x^\prime}D^{m-1}u)_{Q_R^+(X_0)}|^2\,dx\,dt,
\end{align*}
 where $\gamma\in (0,1)$, $0<r<R<\infty$,  $X_0\in Q_R^+$,
and $u$ is the solution of
\begin{equation*}
u_t+(-1)^m\l_0 u:=u_t+(-1)^m \sum_{|\alpha|=|\beta|=m}D^\alpha(A^{\alpha\beta}(t)D^\beta u)=0 \quad \text{in} \,\, Q_{2R}^+
\end{equation*}
with the conormal derivative boundary condition on $\{x_d=0\}\cap Q_{2R}$. We then prove that if the coefficients $A^{\alpha\beta}$ and $f_\alpha$ in \eqref{the system1} are H\"{o}lder continuous in the spatial variables $x$, then  $D_{x^\prime}D^{m-1}u$ is H\"{o}lder continuous in both $t$ and $x$.

In contrast to the Dirichlet boundary condition case in \cite{DZ12}, to estimate $D_d^m u$,  more regularity assumptions on the coefficients and data are necessary. In fact, this is not surprising by  considering the second-order equation
\begin{equation*}
u_t-D_{11}u-D_2(a(t)D_{2}u)=D_2 f(t)\quad \text{in}\quad (0,T)\times\bR_+^2
\end{equation*}
with the conormal derivative boundary condition. The corresponding boundary condition is given by
$$
D_2 u=-f(t)/a(t)\quad \text{on}\quad \{x_2=0\},
$$
which implies that $D_2 u$ is not necessarily continuous. In this paper, we also consider the case  {when} the coefficients are H\"older continuous in both $t$ and $x$. We show that under this stronger assumption all the $m$th order derivatives of $u$ are H\"older continuous in both $t$ and $x$. For the proof, it is sufficient to estimate $D_d^m u$. To this end, we consider a system with special coefficients as in the $L_p$ case, and then use a scaling argument to bound the H\"older semi-norm of $D_d^m u$ by $f_\alpha$ and $D_{x^\prime}D^{m-1}u$.  Combining the estimate of $D_{x^\prime}D^{m-1}u$, which we obtained in the previous case, we are able to estimate $D^m u$.

For linear systems, our Schauder estimate extends the results in Lieberman \cite{L87}, in which the author considered second-order quasilinear parabolic equations with the conormal boundary condition. Schauder estimates for higher-order non-divergence type parabolic systems in the whole space, with the coefficients measurable in $t$ and H\"older continuous in $x$, was considered recently in Boccia \cite{SB12}. With the same class of coefficients, in \cite{DZ12} the authors obtained  Schauder estimates for both divergence type and non-divergence type higher-order parabolic systems in the half space with the Dirichlet boundary condition.  For more results about the conormal derivative problems, we refer the reader to Lieberman \cite{L83,L921} and his new book \cite{L13}.

The paper is organized as follows. In the next section, we introduce some notation and state our main results.  The remaining part of the article can be divided into two parts. In the first part, we treat the $L_p$ estimates.  Section 3 provides some necessary preparations and Section 4 deals with the $L_p$ estimates of $D_d^m u$ for systems with special coefficients.  Section 5 and Section 6 are devoted to the proof of our main result of the $L_p$ estimates.  The second part is about the Schauder estimates.  In Section 7, we establish necessary lemmas and prove the Schauder estimates near a flat boundary with coefficients H\"older continuous only with respect to  $x$. Finally in Section 8, we prove the Schauder estimates when the coefficients are H\"older continuous in both $t$ and $x$.

\section{Main results}
We first introduce some notation used throughout the paper. A point in $\mathbb{R}^d$ is denoted by $x=(x_1,\cdots,x_d)$. We also denote $x=(x^\prime,x_d)$, where $x^\prime \in \mathbb{R}^{d-1}$. A point in
\begin{equation*}
\mathbb{R}^{d+1}=\mathbb{R}\times\mathbb{R}^d=\{(t,x):t\in \mathbb{R},x\in \mathbb{R}^d\}
\end{equation*}
is denoted by $X=(t,x)$.~For $T \in (-\infty,\infty]$, set
\begin{equation*}
\mathcal{O}_T=(-\infty,T)\times \mathbb{R}^d, \quad \mathcal{O}^+_T=(-\infty,T)\times\mathbb{R}^d_+,
\end{equation*}
where $\mathbb{R}^d_+:=\{x=(x_1,\cdots,x_d)\in \mathbb{R}^d: x_d>0\}$. In particular, when $T=\infty$, we use $\mathcal{O}_\infty^+=\bR^{d+1}_+:=\mathbb{R}\times \mathbb{R}^d_+$. Denote
\begin{align*}
&B_r(x_0)=\{y\in \mathbb{R}^d: |x_0-y|<r\},\quad Q_r(t_0,x_0)=(t_0-r^{2m},t_0)\times B_r(x_0),\\
& B_r^+(x_0)=B_r(x_0)\cap \{x_d>0\},\quad Q^+_r(t_0,x_0)=Q_r(t_0,x_0) \cap\mathcal{O}_\infty^+ .
\end{align*}
We use the abbreviation, for instance, $Q_r$ to denote the parabolic cylinder centered at $(0,0)$. The parabolic boundary of $Q_r(t_0,x_0)$ is defined to be
\begin{equation*}
\partial_p Q_r(t_0,x_0)=[t_0,t_0-r^{2m}) \times \partial B_r(x_0)\cup \{t=t_0-r^{2m}\} \times B_r(x_0).
\end{equation*}
The parabolic boundary of $Q^+_r(t_0,x_0)$ is
\begin{equation*}
\partial_p Q^+_r(t_0,x_0)=(\partial_p Q_r(t_0,x_0)\cap \overline{\mathcal{O}_\infty^+} )\cup(Q_r(t_0,x_0)\cap \{x_d =0\}).
\end{equation*}
We denote
\begin{equation*}
\langle f,g\rangle_\Omega=\int_\Omega f^{tr}g=\sum_{j=1}^n\int_\Omega f^jg^j.
\end{equation*}

For a function $f$ on $\mathcal{D} \subset \mathbb{R}^{d+1}$, we set
\begin{equation*}
[f]_{a,b,\mathcal{D}}\equiv \sup\left\{\frac{|f(t,x)-f(s,y)|}{|t-s|^{a}+|x-y|^b} :(t,x),(s,y)\in \mathcal{D}, (t,x)\neq (s,y) \right\},
\end{equation*}
where $a,b\in (0,1]$. The H\"{o}lder semi-norm with respect to $t$ is denoted by
\begin{equation*}
\langle f\rangle_{a,\mathcal{D}}\equiv \sup \left\{\frac{|f(t,x)-f(s,x)|}{|t-s|^{a}}:(t,x),(s,x)\in \mathcal{D}, t\neq s\right\},
\end{equation*}
where $a\in (0,1]$. We also define the H\"older semi-norm with respect to $x$
\begin{equation*}
[f]_{a,\mathcal{D}}^\ast \equiv \sup\left\{\frac{|f(t,x)-f(t,y)|}{|x-y|^a} :(t,x),(t,y)\in \mathcal{D}, x\neq y \right\},
\end{equation*}
and denote
\begin{equation*}
\|f\|^\ast_{a,\mathcal{D}}=\|f\|_{L_\infty(\mathcal{D})}+[f]^\ast_{a,\mathcal{D}},
\end{equation*}
where $a\in(0,1]$. The space corresponding to $\|\cdot\|_a^\ast$ is denoted by  $C^{a\ast }(\mathcal{D})$. When $0<a \le 1$, let
\begin{equation*}
\|f\|_{\frac{a}{2m},a,\mathcal{D}}=\|f\|_{L_\infty(\mathcal{D})}+[f]_{\frac{a}{2m},a,\mathcal{D}}.
\end{equation*}
The space corresponding to $\|\cdot\|_{\frac{a}{2m}, a,\mathcal{D}}$ is denoted by $C^{\frac{a}{2m},a}(\mathcal{D})$. For $1<a<2m$ not an integer, we define
\begin{equation*}
\|f\|_{\frac{a}{2m},a,\mathcal{D}}=\|f\|_{L_\infty(\mathcal{D})}+\sum_{|\alpha|< a}[D^\alpha f]_{\frac{a-|\alpha|}{2m},\{a\},\mathcal{D}},
\end{equation*}
where $\{a\}=a-[a]$.

We denote the average of $f$ in $\cD$ to be
\begin{equation*}
(f)_\mathcal{D}=\frac{1}{|\mathcal{D}|}\int_{\mathcal{D}} f(t,x)\,dx\,dt=\dashint_\mathcal{D}f(t,x)\,dx\,dt.
\end{equation*}
Sometimes we take average only with respect to $x$. For instance,
\begin{equation*}
  (f)_{B_R(x_0)}(t)=\dashint_{B_R(x_0)}f(t,x)\,dx.
 \end{equation*}

Throughout this paper, we assume that all the coefficients are measurable and bounded:
\begin{equation*}
|A^{\alpha\beta}|\le K.
\end{equation*}
In addition, we impose the strong ellipticity condition with a constant $\delta>0$ on the leading coefficients, i.e.,
\begin{equation}\label{strong}
 \sum_{|\alpha|=|\beta|=m}A^{\alpha\beta} \xi_\alpha\xi_\beta\ge \delta \sum_{|\alpha|=m}|\xi_\alpha|^2,
\end{equation}
where $\xi_\alpha \in \mathbb{R}^n$.  We note that the strong ellipticity condition, which is  mainly used in proving Theorem \ref{L_2}, is stronger than the Legender--Hadamard ellipticity condition \eqref{hadamard}.
Here we call $A^{\alpha\beta}$ the leading coefficients if $|\alpha|=|\beta|=m $. All the other coefficients are called lower-order coefficients.
In order to state and prove our results in Sobolev spaces, in addition to the well-known $L_p$ and $W_p^k$ spaces, we introduce the following function spaces. If $\Omega =\mathbb{R}^d$, let
\begin{equation*}
\h_p^m((S,T)\times \mathbb{R}^d)=(1-\Delta)^{\frac{m}{2}}W_p^{1,2m}((S,T)\times \mathbb{R}^d)
\end{equation*}
equipped with the norm
\begin{equation*}
\|u\|_{\h_p^m((S,T)\times\mathbb{R}^d)}=\|(1-\Delta)^{-\frac{m}{2}}u\|_{W_p^{1,2m}((S,T)\times\mathbb{R}^d)}.
\end{equation*}
Here
\begin{equation*}
W_p^{1,2m}((S,T)\times \mathbb{R}^d)=\{u: u_t,D^\alpha u \in L_p((S,T)\times \mathbb{R}^d), 0\le |\alpha|\le 2m\}
\end{equation*}
equipped with its natural norm. Notice that if we set
\begin{align*}
&\mathbb{H}_p^{-m}((S,T)\times \mathbb{R}^d)=(1-\Delta)^{\frac{m}{2}}L_p((S,T)\times \mathbb{R}^d),\\
&\|f\|_{\mathbb{H}_p^{-m}((S,T)\times \mathbb{R}^d)}=\|(1-\Delta)^{-\frac{m}{2}}f\|_{L_p((S,T)\times \mathbb{R}^d)},
\end{align*}
then
\begin{equation*}
\|u\|_{\h_p^m((S,T)\times\mathbb{R}^d)} \cong \|u_t\|_{{\bH}_p^{-m}((S,T)\times \mathbb{R}^d)}+\sum_{|\alpha|\le m}\|D^\alpha u\|_{L_p((S,T)\times \mathbb{R}^d)}.
\end{equation*}
For a general domain $\Omega\subset \bR^d$, we set
\begin{equation*}
\mathbb{H}_p^{-m}((S,T)\times \Omega)=\bigg\{ f:f=\sum_{|\alpha|\le m}D^\alpha f_\alpha, \,f_\alpha\in L_p((S,T)\times \Omega)\bigg\},
\end{equation*}
\begin{equation*}
\|f\|_{\mathbb{H}_p^{-m}((S,T)\times \Omega)}=\text{inf}\bigg\{   \sum_{|\alpha|\le m}\|f_\alpha\|_{L_p((S,T)\times \Omega)}: f=\sum_{|\alpha|\le m}D^\alpha f_\alpha\bigg\},
\end{equation*}
and
\begin{align*}
&\h_p^m((S,T)\times \Omega)\\
&=\{u: u_t \in \mathbb{H}_p^{-m}((S,T)\times \Omega), D^\alpha u \in L_p((S,T)\times \Omega) ,0\le |\alpha|\le m\},\\
&\|u\|_{\h_p^m((S,T)\times \Omega)}=\|u_t\|_{\mathbb{H}_p^{-m}((S,T)\times \Omega)}+\sum_{|\alpha|\le m}\|D^\alpha u\|_{L_p((S,T)\times \Omega)}.
\end{align*}
Let $\mathcal{Q}=\{Q_r(t,x)\cap \o_\infty^+:(t,x)\in \o_\infty^+, r\in(0,\infty) \}$. For a function $g$ defined on $\o_\infty^+$, we denote its (parabolic) maximal and sharp functions, respectively, by
\begin{align*}
&Mg(t,x)=\sup_{Q\in \mathcal{Q}:(t,x)\in Q}\dashint_Q|g(s,y)|\,dy\,ds,\\
&g^\#(t,x) =\sup_{Q\in \mathcal{Q}:(t,x)\in Q}\dashint_Q|g(s,y)-(g)_Q| \,dy\,ds.
\end{align*}
Then
\begin{equation*}
\|g\|_{L_p}\le N\|g^\#\|_{L_p},\quad \|Mg\|_{L_p}\le N\|g\|_{L_p},
\end{equation*}
if $g\in L_p$, where $1<p<\infty$ and $N=N(d,p)$. As is well known, the first inequality above is due to the Fefferman--Stein theorem on sharp functions and the second one is the Hardy--Littlewood maximal function theorem.

Now we state our regularity assumption  on the leading coefficients for the $L_p$ estimates. Let
\begin{equation*}
\text{osc}_x(A^{\alpha\beta}, Q^+_r(t,x))=\dashint_{t-r^{2m}}^{~~~ t} \dashint_{B^+_r(x)}|A^{\alpha\beta}(s,y)-\dashint_{B_r^+(x)}A^{\alpha\beta}(s,z)\,dz|\,dy\,ds,
\end{equation*}
which is the mean oscillation in the spatial variables.
Then we set
\begin{equation*}
A_R^\#=\sup_{(t,x)\in \o_\infty^+}\sup_{r\le R}\sup_{|\alpha|=|\beta|=m}\text{osc}_x(A^{\alpha\beta},Q_r^+(t,x)).
\end{equation*}
We impose on the leading coefficients a small mean oscillation condition with a parameter $\rho>0$, which is specified later.
\begin{assumption}\label{main assumption}
($\rho$). There is a constant $R_0\in(0,1]$ such that $A_{R_0}^\#\le \rho$.
\end{assumption}

We are now ready to present our main results. The first one is the $L_p$ estimates.
\begin{theorem}\label{theorem tan}
Let $p\in(1,\infty)$, $\partial\Omega\in C^{m-1,1}$, $T\in (-\infty,\infty]$, and $f_\alpha\in L_p((-\infty, T)\times \Omega)$ for $|\alpha|\le m$. {The operator $L$ satisfies \eqref{strong}}. Then there exist constants $\rho=\rho(d,n,m, \delta,p, K), $ and $\lambda_0=\lambda_0(d,n,m,\delta,p, K, R_0)$ such that, under Assumption \ref{main assumption} $(\rho)$,  for any $u\in \h_p^m((-\infty, T)\times \Omega)$ satisfying
\begin{equation}
u_t+(-1)^m\l u+\lambda u=\sum_{|\alpha|\le m}D^\alpha f_\alpha \quad \text{in}\quad (-\infty, T)\times \Omega \label{the equation}
\end{equation}
with the conormal derivative boundary condition on ${(-\infty,T)\times}\partial \Omega$,
we have
\begin{equation}
\sum_{k=0}^m\lambda^{1-\frac{k}{2m}}\|D^ku\|_{L_p((-\infty, T)\times \Omega)}\le N\sum_{|\alpha|\le m}\lambda^{\frac{|\alpha|}{2m}}\|f_\alpha\|_{L_p((-\infty, T)\times \Omega)},\label{the estimate}
\end{equation}
provided that $\lambda\ge \lambda_0$, where $N$ depends only on $d,n,m,\delta,p,K$, and $\Omega$.

Moreover, if $\lambda>\lambda_0$, there exists a unique solution $u\in \h_p^m((-\infty, T)\times \Omega)$ of \eqref{the equation}, which satisfies \eqref{the estimate}.
\end{theorem}

The following result is regarding the Schauder estimates near a flat boundary with coefficients only measurable in the $t$ variable.
\begin{theorem}\label{tan thm}
Assume that $a\in(0,1)$, $u\in C_{loc}^{\infty}(\overline{{\o_\infty^+}})$ satisfying
\begin{equation*}
u_t+(-1)^m\l u=\sum_{|\alpha|\le m}D^\alpha f_\alpha\quad \text{in}\quad Q_{4R}^+
\end{equation*}
with the conormal derivative boundary condition on $\{x_d=0\}\cap Q_{4R}$, $f_\alpha \in C^{a\ast}$ if $|\alpha| =m$, $f_\alpha \in L_\infty$ if $|\alpha|< m$, and $A^{\alpha\beta}\in C^{a\ast}$. {The operator $L$ satisfies \eqref{strong}}. Then for any $R\le 1$, there exists a constant $N=N(d,n,m,\delta,K,{a,}\|A^{\alpha\beta}\|_a^\ast,R)$ such that
\begin{align*}
&[D_{x^\prime}D^{m-1} u]_{\frac{a}{2m},a,Q_R^+}\\
&\le N(\sum_{|\alpha|\le m}\|D^\alpha u\|_{L_\infty(Q_{4R}^+)}+\sum_{|\alpha|=m}[f_\alpha]_{a,Q_{4R}^+}^\ast+\sum_{|\alpha|<m}\|f_\alpha\|_{L_\infty(Q_{4R}^+)}).
\end{align*}
\end{theorem}

Our last result is regarding the Schauder estimates in cylindrical domains with more regularity assumptions on the coefficients.
\begin{theorem}\label{thm schauder}
Assume that $a\in (0,1)$, $\partial\Omega \in C^{m,a}, T\in (0,\infty)$, $f_\alpha\in C^{\frac{a}{2m},a}$ if $|\alpha|=m$ and $f_\alpha \in L_\infty$ if $|\alpha|<m$. The operator $\l $ satisfies \eqref{strong}, and $A^{\alpha \beta} \in C^{\frac{a}{2m},a}$. Then the equation
\begin{align*}
&u_t+(-1)^m \l u =\sum_{|\alpha|\le m}D^\alpha f_\alpha \quad \text{in} \quad (0,T)\times \Omega
\end{align*}
with the conormal derivative boundary condition on $(0,T)\times\partial\Omega$ and the zero initial condition on $\{0\}\times \Omega$ has a unique solution  $u\in C^{\frac{a+m}{2m},a+m}((0,T)\times\Omega)$. Moreover, there exists a constant $N=N(d,n,m,\delta,K, \|A^{\alpha\beta}\|_{\frac{a}{2m},a},\Omega,T,a)$ such that,
\begin{equation}\label{finaldiv}
\|u\|_{\frac{a+m}{2m},a+m,(0,T)\times \Omega} \le N\big(\|u\|_{L_2((0,T)\times \Omega)}+G\big),
\end{equation}
where
$$
G=\sum_{|\alpha|=m}{[f_\alpha]}_{\frac{a}{2m},a,(0,T)\times \Omega}+\sum_{|\alpha|\le m}\|f_\alpha\|_{L_\infty((0,T)\times \Omega)}.
$$
\end{theorem}

\section{Some auxiliary estimate{s}}
In this section we consider operators without lower-order terms. Denote
\begin{equation*}
\l_0 u=\sum_{|\alpha|=|\beta|=m}D^\alpha(A^{\alpha\beta}D^\beta u),
\end{equation*}
where $A^{\alpha\beta}=A^{\alpha\beta}(t)$.
Let $C_0^\infty(\overline{\o_T^+})$ be the collection of infinitely differentiable functions defined on $\overline{\o_T^+}$ vanishing for large $|(t,x)|$.

The following $L_2$ estimate for parabolic operators in the divergence form with measurable coefficients is classical.
\begin{theorem}\label{L_2}
There exists a constant $N=N(d,m,n,\delta)$ such that for any $\lambda \ge 0$,
\begin{equation}
\sum_{|\alpha|\le m} \lambda^{1-\frac{|\alpha|}{2m}}\|D^\alpha u\|_{L_2(\o_T^+)}\le N\sum_{|\alpha|\le m}\lambda^{\frac{|\alpha|}{2m}}\|f_\alpha\|_{L_2(\o_T^+)},\label{L_2 simple}
\end{equation}
provided that $u\in \h_2^m(\o_T^+)$ satisfies
\begin{equation}
u_t+(-1)^m \l_0 u+\lambda u=\sum_{|\alpha|\le m}D^\alpha f_\alpha \label{L2 half}
\end{equation}
in $\o_T^+$ with the conormal derivative boundary condition on $\{x_d=0\}$, and $f_\alpha \in L_2(\o_T^+)$, $|\alpha|\le m$. Furthermore, for any $\lambda>0$ and $f_\alpha \in L_2(\o_T^+)$, there exists a unique solution $u \in \h_2^m(\o_T^+)$ to the system \eqref{L2 half} with the conormal derivative boundary condition on $\{x_d=0\}$.
\end{theorem}
\begin{proof}
We assume $\lambda>0$. If $\lambda=0$ the inequality \eqref{L_2 simple} holds trivially or we obtain
\begin{equation*}
\sum_{|\alpha|=m}\|D^\alpha u\|_{L_2(\o_T^+)}\le N\sum_{|\alpha|= m}\|f_\alpha\|_{L_2(\o_T^+)}\quad \text{if}\,\, f_\alpha=0 \,\,\text{for}\,\, |\alpha|<m
\end{equation*}
using the inequality \eqref{L_2 simple} for $\lambda>0$ and letting $\lambda \searrow 0$.

If we have proved the inequality \eqref{L_2 simple}, then due to the fact that
\begin{equation*}
u_t=-(-1)^m \sum_{|\alpha|=|\beta|=m}D^\alpha(A^{\alpha\beta}D^\beta u)-\lambda u+\sum_{|\alpha|\le m}D^\alpha f_\alpha,
\end{equation*}
we obtain
$$\|u\|_{\h_2^m(\o_T^+)}\le N \sum_{|\alpha|\le m}\|f_\alpha\|_{L_2(\o_T^+)},$$
where $N=N(d,n,m,\delta,\lambda)$. Then using this estimate, the method of continuity, and the unique solvability of system with coefficients $A^{\alpha\beta}=\delta_{\alpha\beta}I_{n\times n}$ we prove the second assertion of the theorem. Hence, it is clear that we only need to prove the inequality \eqref{L_2 simple}. Moreover, by a density argument it is obvious that we can assume that $u \in C_0^\infty(\overline{\o_T^+})$.

By the weak formulation of the conormal derivative boundary condition, we have
\begin{align*}
&\langle u,u_t\rangle_{\o_T^+}+\sum_{|\alpha|= |\beta|= m}\langle D^\alpha u, A^{\alpha\beta}D^\beta u\rangle_{\o_T^+}+\lambda \langle u,u\rangle_{\o_T^+}\\
&=\sum_{|\alpha|\le m}(-1)^{|\alpha|} \langle D^\alpha u,f_\alpha \rangle_{\o_T^+}.
\end{align*}
From the strong ellipticity {condition}, we get
\begin{equation*}
\delta \int_{\o_T^+}|D^m u|^2 \,dx\,dt\le\sum_{|\alpha|=|\beta|=m} \langle D^\alpha u, A^{\alpha\beta}D^\beta u\rangle_{\o_T^+}.
\end{equation*}
Moreover,
\begin{align*}
2\langle u, u_t\rangle_{\o_T^+}=\int_{\o_T^+}\frac{\partial}{\partial t}|u|^2(t,x)\,dx\,dt =\int_{\mathbb{R}^d_+}|u|^2(T,x) \,dx\ge 0.
\end{align*}
Hence for any $\varepsilon >0$, by Young's inequality,
\begin{align*}
&\delta \int_{\o_T^+}|D^m u|^2 \,dx\,dt+\lambda \int_{\o_T^+}|u|^2\,dx\,dt \le \sum_{|\alpha|\le m}(-1)^{|\alpha|}\int_{\o_T^+}D^\alpha u f_\alpha\,dx\,dt\\
&\le \varepsilon \sum_{|\alpha|\le m}\lambda^{\frac{m-|\alpha|}{m}}\int_{\o_T^+}|D^\alpha u|^2\,dx\,dt +N\varepsilon^{-1}\sum_{|\alpha|\le m}\lambda^{-\frac{m-|\alpha|}{m}} \int_{\o_T^+}|f_\alpha|^2 \,dx\,dt.
\end{align*}
To finish the proof, it suffices to use interpolation inequalities and choose $\varepsilon$ sufficiently small depending on $\delta$, $d$, $m$, and $n$.
\end{proof}

By Theorem \ref{L_2} and adapting the proofs of Lemmas 3.2 and 7.2 in \cite{DK11} to the conormal case, we have the following local $L_2$ estimates.
\begin{lemma}\label{L2 local}
Let $0<r<R<\infty$. Assume that $u \in C^\infty_{loc}(\overline{{\mathbb{R}^{d+1}_+}})$ satisfies
\begin{equation}
u_t+(-1)^m\l_0 u=0 \quad \text{in}\quad Q_R^+ \label{local L_2}
\end{equation}
with the conormal derivative boundary condition on $\{x_d=0\}\cap Q_R$. Then there exists a constant $N=N(d,m,n,\delta)$ such that for $j=1,\ldots,m$,
\begin{equation*}
\|D^j u\|_{L_2(Q_r^+)}\le N(R-r)^{-j}\|u\|_{L_2(Q_R^+)}.
\end{equation*}
\end{lemma}
\begin{corollary}
            \label{L_2 high order}
Let $0<r<R<\infty$ and assume that $u \in C^\infty_{loc}(\overline{{\mathbb{R}^{d+1}_+}})$ satisfies \eqref{local L_2} with the conormal derivative boundary condition on  $\{x_d=0\}\cap Q_R$. Then for any multi-index $\theta$ satisfying $\theta_d\le m$, we have
\begin{equation*}
\|D^\theta u\|_{L_2(Q_r^+)}\le N\|u\|_{L_2(Q_R^+)},
\end{equation*}
where $N=N(d,m,n,\delta,R,r,\theta)$.
\end{corollary}
 \begin{proof}
 Obviously, $D_{x^\prime}^{k}u$, $k=1,2,\ldots$ satisfies the same equation and boundary condition as $u$. After applying Lemma \ref{L2 local} repeatedly, we immediately see that
 \begin{equation*}
 \|D^jD_{x^\prime}^{k} u\|_{L_2(Q_{r}^+)}\le N\|u\|_{L_2(Q_R^+)},\quad j=1,2,\ldots,m.
 \end{equation*}
 The corollary is proved.
 \end{proof}

 \begin{lemma}\label{tech 3}
Let $R\in (0,\infty)$. Assume that $u\in C_{loc}^\infty(\overline{\o_\infty^+})$ satisfies
\begin{equation*}
u_t+(-1)^m \l_0 u=\sum_{|\alpha|\le m}D^\alpha f_\alpha \quad \text{in}\quad Q_{2R}^+
\end{equation*}
with the conormal derivative boundary condition on $\{x_d=0\}\cap Q_{2R}$. Let $P(x)$ be a vector-valued polynomial of order $m-1$ and satisf{y}
\begin{equation*}
(D^\alpha P(x))_{Q_R^+}=(D^\alpha u)_{Q_R^+}
\end{equation*}
for  $|\alpha|<m$. Let $v=u-P(x)$. Then there exists a constant $N=N(d,n,m)$ such that
\begin{align*}
&\|D^\alpha v\|_{L_2(Q_R^+)}\\
&\le NR^{m-|\alpha|}\|D^m u\|_{L_2(Q_R^+)}+N\sum_{|\beta|\le m}R^{3m+d/2-|\alpha|-|\beta|}\|f_\beta\|_{L_\infty(Q_R^+)},
\end{align*}
where $|\alpha|<m$. If $f_\alpha=0$ for any $\alpha$, then the following inequality hold{s}
\begin{equation*}
\|D^\alpha v\|_{L_2(Q_R^+)}\le NR^{|\beta|-|\alpha|}\|D^\beta v\|_{L_2(Q_R^+)},
\end{equation*}
where $|\alpha|<|\beta|\le m$.
\end{lemma}
\begin{proof}
By using a scaling argument, we only need to consider the case when $R=1$. Choose $\xi(y)\in C_0^\infty(B_1^+)$ with unit integral. Then set
\begin{equation*}
g_\alpha(t)=\int_{B_1^+}D^\alpha v(t,y)\xi(y)\,dy,\quad c_\alpha=\int_{-1}^0g_\alpha(t)\,dt
\end{equation*}
for $|\alpha|<m$. By H\"{o}lder's inequality and the Poincar\'e inequality, the following estimate holds
 \begin{align}\nonumber
& \int_{B^+_1}| D^\alpha v(t,x) -g_\alpha(t)|^2 \,dx=\int_{B^+_1}|\int_{B^+_1}( D^\alpha v(t,x)- D^\alpha v(t,y))\xi(y)\,dy|^2dx\\
&\le N\int_{B^+_1}\int_{B^+_1}|D^\alpha v(t,x)-D^\alpha v(t,y)|^2\,dy\,dx \le N\int_{B^+_1}|D^{|\alpha|+1} v(t,y)|^2\,dy.\nonumber
 \end{align}
Because $(D^\alpha v)_{Q^+_1}=0$, by the triangle inequality and the Poincar\'e inequality
 \begin{align}\nonumber
 &\|D^\alpha v\|_{L_2(Q^+_1)}\le \|D^\alpha v-c_\alpha\|_{L_2(Q^+_1)} \\\nonumber
 &\le\|D^\alpha v-g_\alpha(t)\|_{L_2(Q^+_1)}+\|g_\alpha(t)-c_\alpha\|_{L_2(Q^+_1)}\\
 &\le N\|D^{|\alpha|+1}v\|_{L_2(Q^+_1)}+N\|\partial_t g_\alpha\|_{L_2(-1,0)}.\label{tech3.2}
 \end{align}
Since {$v$ satisfies the same equations as $u$}, by the definition of $g_\alpha$,
\begin{align*}
\partial_tg_\alpha(t)&=\int_{B^+_1}\xi(y)D^\alpha\partial_t {v}(t,y)\,dy\\
&=\int_{B^+_1}(-1)^{m+1}\xi(y)D^\alpha\l_0{v}(t,y)\,dy+\sum_{|\beta|\le m}\int_{B^+_1}\xi(y)D^\alpha D^\beta f_\beta dy.
\end{align*}
For the first term of the right-hand side of the equality above, we leave ${|\alpha|+1}$ derivatives on ${v}$ and move all the others to $\xi$. For the second term, we move all the derivatives to $\xi$.
Therefore, by the Cauchy--Schwarz inequality,
\begin{align}
|\partial_t g_\alpha(t)|&\le N \int_{B_1^+}|{D^{|\alpha|+1}v}|\, dy+N\sum_{|\alpha|\le m}\|f_\alpha\|_{L_\infty(Q_1^+)}\nonumber\\
&\le N\|{D^{|\alpha|+1}v}(t,\cdot)\|_{L_2(B_1^+)}+N\sum_{|\alpha|\le m}\|f_\alpha\|_{L_\infty(Q_1^+)}.\label{tech 3.3}
\end{align}
Combining \eqref{tech 3.3} and \eqref{tech3.2}, we prove the desired estimate for $R=1$ by induction.
\end{proof}

\section{$L_p$ estimate of $D_d^m u$ for systems with special coefficients}
{In this section} we consider the following special system:
\begin{equation*}
\tilde{\l_0} =D^{\hat{\alpha}} (A^{\hat{\alpha}\hat{\alpha}}(t)D^{\hat{\alpha}})+ \cD^m_{d-1}I_{n\times n},
\end{equation*}
where
\begin{equation*}
\hat{\alpha}=(0,\cdots,0, m),\quad \cD^m_{d-1}=\sum_{j=1}^{d-1}D_j^{2m}.
\end{equation*}
\begin{lemma}\label{L2 special}
There exists a constant $N=N(d,m,n,\delta)$ such that for any $\lambda \ge 0$,
\begin{equation*}
\sum_{|\alpha|\le m} \lambda^{1-\frac{|\alpha|}{2m}}\|D^\alpha u\|_{L_2(\o_T^+)}\le N\sum_{|\alpha|\le m}\lambda^{\frac{|\alpha|}{2m}}\|f_\alpha\|_{L_2(\o_T^+)}
\end{equation*}
provided that $u\in \h_2^m(\o_T^+)$ satisfies
\begin{equation*}
u_t+(-1)^m \tilde{\l}_0 u+\lambda u=\sum_{|\alpha|\le m}D^\alpha f_\alpha 
\end{equation*}
in $\o_T^+$ with the conormal derivative boundary condition on $\{x_d=0\}$, and $f_\alpha \in L_2(\o_T^+)$, $|\alpha|\le m$. Furthermore, for any $\lambda>0$ and $f_\alpha \in L_2(\o_T^+)$, $|\alpha|\le m$, there exists a unique solution $u \in \h_2^m(\o_T^+)$ to the system \eqref{L2 half} with the conormal derivative boundary condition on $\{x_d=0\}$.
\end{lemma}
\begin{proof}
We follow the proof of Theorem \ref{L_2 simple} and only need to show an a priori estimate. By the weak formulation of the conormal derivative boundary condition,
\begin{align*}
\langle u,u_t\rangle_{\o_T^+}+\sum_{j=1}^{d-1}\langle D_j^m u, D_j^m u\rangle_{\o_T^+}+\langle D_d^mu, A^{\hat{\alpha}\hat{\alpha}}D^m_d u\rangle+\lambda \langle u,u\rangle_{\o_T^+}\\
=\sum_{|\alpha|\le m}(-1)^{|\alpha|} \langle D^\alpha u,f_\alpha \rangle_{\o_T^+}.
\end{align*}
 It is obvious that $A^{\hat{\alpha}\hat{\alpha}}\ge \delta I_{n\times n}$. Hence
$$\sum_{j=1}^{d-1}\langle D_j^m u, D_j^m u\rangle_{\o_T^+}+\langle D_d^mu, A^{\hat{\alpha}\hat{\alpha}}D^m_d u\rangle\ge \delta \sum_{j=1}^d\|D_j^mu\|^2_{L_2(\o_T^+)}.$$
From Proposition 1 of \cite{DK10}, we know that
\begin{equation*}
\|D^m u\|_{L_2(\o_T^+)}\le N\sum_{j=1}^d\|D_j^m u\|_{L_2(\o_T^+)}.
\end{equation*}
Therefore,
$$\sum_{j=1}^{d-1}\langle D_j^m u, D_j^m u\rangle_{\o_T^+}+\langle D_d^mu, A^{\hat{\alpha}\hat{\alpha}}D^m_d u\rangle\ge \frac{\delta}{N}\|D^mu\|^2_{L_2(\o_T^+)}.$$
The rest of the proof just follows the proof of Theorem \ref{L_2}, and thus is omitted.
\end{proof}

\begin{remark}
                            \label{rem4.2}
From the proof of Lemma \ref{L2 special}, we see that Theorem \ref{L_2} still holds when the strong ellipticity is replaced by a weak condition:
\begin{equation}
                                    \label{eq3.06}
\delta \sum_{j=1}^{d}\langle D_j^m u, D_j^m u\rangle_{\o_T^+}\le\sum_{|\alpha|=|\beta|=m} \langle D^\alpha u, A^{\alpha\beta}D^\beta u\rangle_{\o_T^+}.
\end{equation}
for any $u\in C_0^\infty(\overline{\o_T^+})$.
\end{remark}

We have the following observation.
\begin{lemma}\label{boundary condition}
Assume $u\in C_{loc}^\infty(\overline{{\o_\infty^+}})$ and satisfies
\begin{equation}
u_t+(-1)^m \tilde{\l_0} u=0\label{specialest}
\end{equation}
in $Q_4^+$ with the conormal derivative boundary condition on $\{x_d=0\}\cap Q_4$. Then the boundary condition is given by
\begin{equation*}
D_d^m u=\cdots=D_d^{2m-1}u=0\quad \text{on}\quad \{x_d=0\}\cap Q_4.
\end{equation*}
\end{lemma}
\begin{proof}
Following the definition, the weak formulation of the system is
\begin{align*}
-\int_{Q_4^+}u \phi_t \,dx\,dt+\sum_{j=1}^{d-1}\int_{Q_4^+}D_j^muD_j^m\phi\,dx\,dt+\int_{Q_4^+} A^{\hat{\alpha}\hat{\alpha}}(t)D^m_du D^m_d\phi \,dx\,dt=0,\nonumber
\end{align*}
where $\phi\in C_0^{\infty}(\overline{Q_4^+})$.
Since only the boundary condition on $\{x_d=0\}$ is considered, we integrate by parts and boundary terms appear in the last term of the equation above. Let us denote $\Sigma:= \{x_d=0\}\cap Q_4$. We integrate by parts repeatedly and the boundary term is
\begin{equation*}
\int_\Sigma \sum_{j=1}^m (-1)^{j+1}A^{\hat{\alpha}\hat{\alpha}}D_d^{m+j-1} u \, D_d^{m-j} \phi \,dx^\prime \,dt=0.
\end{equation*}
Since $\phi$ is an arbitrary smooth function and $A^{\hat{\alpha}\hat{\alpha}}$ is positive definite, we find that
\begin{equation*}
A^{\hat{\alpha}\hat{\alpha}}D_d^{m+j-1}u=0\quad\text{on}\quad \Sigma \quad  \text{for}\quad j=1,2,\ldots, m,
\end{equation*}
i.e.,
\begin{equation*}
D_d^m u=\ldots=D_d^{2m-1}u=0\quad \text{on}\quad {\Sigma}.
\end{equation*}
The lemma is proved.
\end{proof}

We state a conclusion in Remark 6 of \cite{DK10} and notice that the operator ${\tilde\l_0}$ satisfies the Legendre--Hadamard condition \eqref{hadamard}.
\begin{lemma}\label{rk dz}
Assume that $u\in C_{loc}^\infty(\overline{{\mathbb{R}^{d+1}_+}})$ and satisfies
\begin{align*}
&u_t+(-1)^m {\tilde\l_0} u=0\quad \text{in}\quad Q_{2R}^+,\\
&u=\cdots=D_d^{m-1}u=0\quad \text{on}\quad Q_{2R}\cap\{x_d=0\}.
\end{align*}
Then for any $0<r<R<\infty$, there exists a constant $N$ depending on $d$, $n$, $m$, $K$, $\delta$, $r$, $R$, and $a$ such that
\begin{equation*}
\|u\|_{\frac{a}{2m},a, Q_r^+}\le N\|u\|_{L_2(Q_R^+)},
\end{equation*}
where $0<a<2m$.
\end{lemma}

Next we prove the following lemma.

\begin{lemma}\label{homo mean}
Let $\lambda\ge 0$.  Assume that $u\in C_{loc}^\infty(\overline{{\o_\infty^+}})$ satisfies
\begin{equation}
u_t+(-1)^m\tilde{\l_0} u+\lambda u=0\label{lambda equation}
\end{equation}
in $Q_2^+$ with the conormal derivative boundary condition on $\{x_d=0\}\cap Q_2$. Then there exists $N=N(d,m,n,\delta,K)$ such that
\begin{equation*}
[D_d^m u]_{\frac{1}{2m},1, Q_1^+}\le N\|D_d^m u\|_{L_2(Q_2^+)}.
\end{equation*}
\end{lemma}
\begin{proof}
For the case $\lambda=0$, as noted in Lemma \ref{boundary condition}, the conormal boundary condition for \eqref{specialest} is given by
\begin{equation*}
D_d^mu=\cdots=D_d^{2m-1}u=0\quad \text{on}\quad \{x_d=0\}.
\end{equation*}
We differentiate \eqref{specialest} $m$ times with respect to $x_d$ and let $v=D_d^m u$. Then we arrive at
\begin{align*}
&v_t+(-1)^m \tilde{\l_0} v=0\quad \text{in}\,\,Q_2^+,\\
&v=D_d v=\cdots=D_d^{m-1}v=0\quad \text{on}\,\,\{x_d=0\}\cap Q_2.
\end{align*}
By Lemma \ref{rk dz} with $a=1$,
\begin{equation*}
[v]_{\frac{1}{2m},1,{Q}_1^+}\le N\|v\|_{L_2({Q}_{2}^+)}
\end{equation*}
Since $v=D_d^m u$, we prove the case when $\lambda=0$. For the case when $\lambda>0$,  we apply an idea of S. Agmon, the details of which can be found in Corollary \ref{mean basic}.
\end{proof}

\begin{lemma}\label{mean special1}
Let $r\in (0,\infty)$, $\kappa \in [32,\infty)$, $\lambda>0$, and $X_0=(t_0,x_0)\in \overline{\o_\infty^+}$. Assume that $u \in C_{loc}^\infty(\overline{\mathbb{R}^{d+1}_+})$ satisfies \eqref{lambda equation}
in $Q_{\kappa r}^+(X_0)$ with the conormal derivative boundary condition on $\{x_d=0\}\cap Q_{\kappa r}(X_0)$. Then we have
\begin{equation}
(|D_d^m u-(D_d^m u)_{Q_r^+(X_0)}|)_{Q_r^+(X_0)}\le N\kappa^{-1}(|D_d^m u|^2)^{\frac{1}{2}}_{Q^+_{\kappa r}(X_0)}\label{mean 2},
\end{equation}
where $N=N(d,n,m,\delta,K)$.
\end{lemma}
\begin{proof}
By scaling, it suffices to prove the inequality for $r=8/\kappa$.
We consider the following two cases.

Case 1: the last coordinate of $x_0\in[0,1)$. In this case, by denoting $Y_0=(t_0,x_0^\prime,0)$, we have
\begin{equation*}
Q_r(X_0)\subset Q_2(Y_0)\subset Q_4(Y_0)\subset Q_{\kappa r}(X_0).
\end{equation*}
After applying Lemma \ref{homo mean} to $u$ with a translation of the coordinates, we obtain
\begin{align*}
(|D_d^mu-(D_d^mu)_{Q_r^+(X_0)}|)_{Q_r^+(X_0)}\le N r[D_d^mu]_{\frac{1}{2m}, 1, Q_2^+(Y_0)}\\
\le N \kappa^{-1} (|D_d^m u|^2)_{Q_4^+(Y_0)}^{\frac{1}{2}}\le N\kappa^{-1}(|D^m_d u|^2)_{Q_{\kappa r}^+(X_0)}^{\frac{1}{2}}.
\end{align*}

Case 2: the last coordinate of $x_0\ge 1$.  This case is indeed an interior case. From Lemmas 2 and 3 in \cite{DK10}, we can show
\begin{equation*}
[v]_{\frac{1}{2m}, 1, Q_1}\le N\|v\|_{L_2(Q_4)},
\end{equation*}
 where $v$ smooth is a solution of
 \begin{equation*}
 v_t+(-1)^m \tilde{\l}_0 v+\lambda v=0\quad \text{in}\quad Q_4.
 \end{equation*}
Taking $v=D_d^m u$, we get
\begin{align*}
(|D_d^mu-(D_d^mu)_{Q_r(X_0)}|)_{Q_r(X_0)}\le N r[D_d^mu]_{\frac{1}{2m}, 1, Q_{1/4}(X_0)}\\
\le Nr\|D_d^mu\|_{L_2(Q_{1}(X_0))}
\le N\kappa^{-1}(|D^m_d u|^2)_{Q_{\kappa r}^+(X_0)}^{\frac{1}{2}}.
\end{align*}
Hence we prove the lemma.
\end{proof}

Now we are ready to establish a mean oscillation  estimate of $D_d^m u$ for systems with special coefficients in the half space.
\begin{theorem}\label{mean half space}
Let $r\in (0,\infty)$, $\kappa \in [64,\infty)$, $\lambda>0$, $X_0=(t_0,x_0)\in \overline{\mathbb{R}^{d+1}_+}$, and $f_\alpha \in L_{2, loc}(\mathbb{R}^{d+1}_+)$, where $|\alpha|\le m$. Assume that $u \in C_{loc}^{\infty}(\overline{\mathbb{R}^{d+1}_+})$ satisfies
\begin{equation*}
u_t+(-1)^m\tilde{{\l_0}} u+\lambda u=\sum_{|\alpha|\le m}D^\alpha f_\alpha
\end{equation*}
in $Q_{\kappa r}^+(X_0)$ with the conormal derivative boundary condition on $\{x_d=0\}\cap Q_{\kappa r}(X_0)$. Then we have
\begin{align}
&(|D_d^m u-(D_d^m u)_{Q_r^+(X_0)}|)_{Q_r^+(X_0)}\nonumber\\
&\le N\kappa^{-1}(|D_d^m u|^2)^{\frac{1}{2}}_{Q_{\kappa r}^+(X_0)}+N\kappa^{m+\frac{d}{2}}\sum_{|\alpha|\le m}\lambda^{\frac{|\alpha|}{2m}-\frac{1}{2}}(|f_\alpha|^2)^{\frac{1}{2}}_{Q^+_{\kappa r}(X_0)}\label{mean oscillation special1},
\end{align}
where $N=N(d,m,n,\delta,K)$.
\end{theorem}

\begin{proof}
Choose two smooth functions $\zeta$ and $\zeta_1$ defined on $\bR^{d+1}$ such that
\begin{align*}
\zeta&=1\quad \text{on} \quad Q_{\kappa r/2}(X_0), \\
\zeta&=0 \quad \text{outside}\quad (t_0-(\kappa r)^{2m},t_0+(\kappa r)^{2m})\times B_{\kappa r}(x_0),
\end{align*}
and
\begin{align*}
\zeta_1&=1\quad \text{on} \quad Q_{\kappa r}(X_0), \\
\zeta_1&=0 \quad \text{outside}\quad (t_0-(2\kappa r)^{2m},t_0+(2\kappa r)^{2m})\times B_{2\kappa r}(x_0).
\end{align*}
Since we only concern the values of $u$ and $f_\alpha$ in $Q_{\kappa r}^+(X_0)$, let us consider $\tilde{u}=\zeta_1 u$. By a simple calculation, we can show that $\tilde{u}$ satisfies the following equation in $\o_\infty^+$
\begin{equation*}
\tilde{u}_t+(-1)^m\tilde{\l_0} \tilde{u}+\lambda \tilde{u}=\sum_{|\alpha|\le m}D^\alpha F_\alpha
\end{equation*}
with the conormal derivative boundary condition on $\{x_d=0\}$, where $F_\alpha$ is the linear combination of terms like
$$
u{\zeta_1}_t,\quad D^k\zeta_1 f_\alpha,\quad
A^{\alpha\beta}D^m uD^k \zeta_1,\quad
A^{\alpha\beta}D^{m-k}u D^{k}\zeta_1,\,\, k\ge 1.
$$
Since $f_\alpha\in L_{2, loc}(\o_\infty^+)$, $u$ is smooth, and $\zeta$ has compact support, each term above is in $L_2$, {which implies that} $F_\alpha\in L_2(\o_\infty^+)$. Because $f_\alpha=F_\alpha$ in $Q_{\kappa r}^+(X_0)$, without loss of generality we can assume $f_\alpha\in L_2(\o_\infty^+)$.

By Lemma \ref{L2 special}, for any $\lambda>0$, there exists a unique solution $w \in \h_2^m(\o_\infty^+)$ to the equation
\begin{equation*}
w_t+(-1)^m \tilde{\l_0} w+\lambda w=\sum_{|\alpha|\le m}D^\alpha (\zeta f_\alpha)
\end{equation*}
in $\o_\infty^+$ with the conormal derivative boundary condition on $\{x_d =0\}$. By a mollification argument, we may assume that $w$ is smooth. In fact, let $f_{\alpha\epsilon}$ be smooth functions which converge to $f_\alpha$ in $L_2$ and $A^{\alpha\hat{\alpha}}_\epsilon$ be the mollification of $A^{\alpha\hat{\alpha}}$ with respect to the $t$ variable, which converges to $A^{\hat{\alpha}\hat{\alpha}}$ almost everywhere. We denote the operator to be $\tilde{\l_0}_\epsilon$ and consider the following equation
\begin{equation*}
(u_{\epsilon})_t+(-1)^m \tilde{\l_0}_\epsilon u_\epsilon+\lambda u_\epsilon=\sum_{|\alpha|\le m}D^\alpha f_{\alpha\epsilon}\quad \text{in}\quad \o_\infty^+
\end{equation*}
with the conormal derivative boundary condition on $\{x_d=0\}$.  {By Lemma \ref{L2 special},}  $u_\epsilon$ converges to $u$ in $\h_2^m(\o_\infty^+)$. If \eqref{mean oscillation special1} holds for $u_\epsilon$, we can pass to the limit. Hence we may assume $f_\alpha$ and $A^{\hat{\alpha}\hat{\alpha}}$ are smooth functions, which implies that $w$ is smooth. Let $v:=u-w$. Then the function $v$ is smooth as well and satisfies
\begin{equation*}
v_t+(-1)^m\tilde{\l_0} v+\lambda v=0 \quad \text{in}\quad Q^+_{\kappa r/2}(X_0)
\end{equation*}
with the conormal derivative boundary condition on $\{x_d=0\}\cap Q_{\kappa r/2}(X_0)$. By applying Lemma \ref{mean special1} (note that $\kappa/2\ge 32$) to $v$, we have
\begin{align}
(|D_d^m v-(D_d^m v)_{Q^+_r(X_0)}|)_{Q^+_r(X_0)}
\le N \kappa^{-1}(|D_d^m v|^2)^{\frac{1}{2}}_{Q^+_{\kappa r/2}(X_0)}\label{v estimate}.
\end{align}
 By Lemma \ref{L2 special} with $T=t_0$, we get
\begin{equation}
\sum_{|\alpha|\le m} \lambda^{1-\frac{|\alpha|}{2m}}\|D^\alpha w\|_{L_2(\o_{t_0}^+)}\le N\sum_{|\alpha|\le m}\lambda^{\frac{|\alpha|}{2m}}\|\zeta f_\alpha\|_{L_2(\o_{t_0}^+)}.\label{eq 4.18}
\end{equation}
In particular,
\begin{align}
&(|D^m w|^2)^{\frac{1}{2}}_{Q^+_r(X_0)}
\le N\kappa^{m+d/2}\sum_{|\alpha|\le m}\lambda^{\frac{|\alpha|}{2m}-\frac{1}{2}}(|f_\alpha|^2)_{Q^+_{\kappa r}(X_0)}^{\frac{1}{2}}\label{w estimate}.
\end{align}
Let us prove \eqref{mean oscillation special1} now. By the triangle inequality,
\begin{equation*}
(|D_d^m u-(D_d^m u)_{Q_r^+(X_0)}|)_{Q_r^+(X_0)}\le 2(|D^m_d u-c|)_{Q_r^+(X_0)}
\end{equation*}
for any constant $c$. By taking $c=(D^m_d v)_{Q_r^+(X_0)}$, we have
\begin{align*}
(|D^m_du-(D^m_d u)_{Q_r^+(X_0)}|)_{Q_r^+(X_0)}\le {2} (|D_d^m u-(D_d^m v)_{Q_r^+(X_0)}|)_{Q_r^+(X_0)}.
\end{align*}
By using the triangle inequality and the Cauchy--Schwarz inequality, the right-hand side of the inequality above can be bounded by
\begin{align*}
& N (|D^m_d v -(D^m_d v)_{Q_r^+(X_0)}|)_{Q^+_r(X_0)}+N (|D^m w|^2)^{\frac{1}{2}}_{Q_r^+(X_0)}.
\end{align*}
By using \eqref{v estimate} and \eqref{w estimate}, the quantity above is less than
\begin{equation}
N\kappa^{-1}(|D_d^m v|^2)^{\frac{1}{2}}_{Q^+_{\kappa r/2}(X_0)}+N \kappa^{m+d/2}\sum_{|\alpha|\le m}\lambda^{\frac{|\alpha|}{2m}-\frac{1}{2}}(|f_\alpha|^2)^{\frac{1}{2}}_{Q^+_{\kappa r}(X_0)}.\label{eq 4.20}
\end{equation}
Finally, from \eqref{eq 4.18} we know that
\begin{equation*}
(|D^m_d w|^2)^\.5_{Q_{\kappa r/2}^+(X_0)}\le N \sum_{|\alpha|\le m}\lambda^{\frac{|\alpha|}{2m}-\.5}(|f_\alpha|^2)^\.5_{Q_{\kappa r}^+(X_0)}.
\end{equation*}
Combining the inequality above and the fact that $v=u-w$, we see that the first term of \eqref{eq 4.20} is less than the right-hand side of \eqref{mean oscillation special1}.
\end{proof}

Now we are ready to prove an $L_p$ estimate of $D_d^m u$ for the system with special coefficients.
\begin{theorem}\label{L_p special}
Let $p\in[2,\infty)$,  $\lambda\ge 0$, $T\in(-\infty,\infty]$, and $f_\alpha \in L_p(\o_T^+)$ for $|\alpha|\le m$. Then for any $u\in \h_p^m(\o_T^+)$ satisfying
\begin{equation*}
u_t+(-1)^m \tilde{\l_0} u+\lambda u=\sum_{|\alpha|\le m}D^\alpha f_\alpha\quad \text{in} \,\,\o_T^+
\end{equation*}
with the conormal derivative boundary condition on $\{x_d=0\}$,
we have
\begin{equation*}
\lambda^{\.5}\|D_d^m u\|_{L_p(\o_T^+)}\le N\sum_{|\alpha|\le m}\lambda^{\frac{|\alpha|}{2m}}\|f_\alpha\|_{L_p(\o_T^+)},
\end{equation*}
where $N=N(d,m,n,p,\delta,K)$.
\end{theorem}
\begin{proof}
Due to a density argument, it suffices to assume $u\in C_0^\infty(\overline{\o_\infty^+})$.  First {we suppose that} $p\in(2,\infty)$ {and} $T=\infty$. Under these assumptions, from  Theorem \ref{mean half space}, we deduce that
\begin{align*}
&(D^m_d u)^{\#}(X_0)\\
&\le \kappa^{-1}(M(D_d^m u)^2(X_0))^{\frac{1}{2}}
+N \kappa^{m+\frac{d}{2}}\sum_{|\alpha|\le m}\lambda ^{\frac{|\alpha|}{2m}-\frac{1}{2}}(M(f_\alpha)^2(X_0))^{\frac{1}{2}}
\end{align*}
for any $\kappa\ge 64$ and $X_0\in {\overline{\o_\infty^+}}$. This, together with the Fefferman--Stein theorem {and} the Hardy--Littlewood maximal function theorem, yields
\begin{align*}
&\|D^m_d u\|_{L_p(\o^+_\infty)}\le N\|(D^m_du)^\#\|_{L_p(\o^+_\infty)}\\
&\le N\kappa^{-1}\|(M(D_d^m u)^2)^\.5\|_{L_p(\o^+_\infty)}+N\kappa^{m+\frac{d}{2}}\sum_{|\alpha|\le m}\lambda ^{\frac{|\alpha|}{2m}-\frac{1}{2}}\|(M(f_\alpha)^2)^\.5\|_{L_p(\o^+_\infty)}\\
&\le N\kappa^{-1}\|D_d^m u\|_{L_p(\o^+_\infty)}+N\kappa^{m+\frac{d}{2}}\sum_{|\alpha|\le m}\lambda ^{\frac{|\alpha|}{2m}-\frac{1}{2}}\|f_\alpha\|_{L_p(\o^+_\infty)}.
\end{align*}
Now we choose $\kappa$ sufficiently large such that the first term on the right-hand side of the inequality above is absorbed in the left-hand side. Then we obtain the desired estimate. For $T\in(-\infty,\infty)$, a standard argument in \cite{Kry07} can be applied.
The case $p=2$ follows from Lemma \ref{L2 special}.
\end{proof}

\section{Mean oscillation estimate for {$D_{x'}D^{m-1}u$}}
In this section, we obtain a mean oscillation estimate for $D_{x^\prime}D^{m-1}u$. The following lemma shows that $\|D_d^m u\|_{L_p}$ can be bounded by $\|D_{x^\prime}D^{m-1}u\|_{L_p}$ and $\|f_\alpha\|_{L_p}$ for  systems with simple coefficients.
\begin{lemma}\label{nor and tan}
Let $T\in (-\infty,\infty]$ and $p\in[2,\infty)$. Assume that $u\in \h_p^m(\o_T^+)$ and $$u_t+(-1)^m\l_0 u+\lambda u=\sum_{|\alpha|\le m}D^\alpha f_\alpha$$ with the conormal derivative boundary condition on $\{x_d=0\}$, where $\lambda\ge 0$ and $f_\alpha \in L_p(\o_\infty^+)$. Then there exists a constant $N$, depending only on $d,m,n,\delta, K$, and $p$, such that
\begin{equation*}
\lambda^{\frac{1}{2}}\|D^m u\|_{L_p(\o_T^+)}\le N\sum_{|\alpha|\le m}\lambda^{\frac{|\alpha|}{2m}}\|f_\alpha\|_{L_p(\o_T^+)}
+N\lambda^{\frac{1}{2}}\|D_{x^{\prime}}D^{m-1}u\|_{L_p(\o_T^+)}.
\end{equation*}
Especially, if $\lambda=0$ and $f_\alpha=0$ when $|\alpha|<m$, then
\begin{equation*}
\|D^m u\|_{L_p(\o_T^+)}\le N\|D_{x^\prime }D^{m-1}u\|_{L_p(\o_T^+)}+N\sum_{|\alpha|=m}\|f_\alpha\|_{L_p(\o_T^+)}.
\end{equation*}
\end{lemma}
\begin{proof}
The case $\lambda=0$ follows by letting $\lambda\searrow 0$ after the estimate with $\lambda >0$ is proved.

We use a scaling argument. Let $v(t,x^\prime, x_d)=u(\mu^{-2m}t, x^\prime, \mu^{-1} x_d)$ with a sufficiently large constant $\mu$ to be chosen later. Then $v$ satisfies in $\o_{\mu^{2m}T}^+$,
 \begin{align*}
&v_t+(-1)^m\sum_{|\alpha|=|\beta|=m}\mu^{\alpha_d+\beta_d-2m}D^\alpha (\tilde{A}^{{\alpha}{\beta}}D^\beta v){+\mu^{-2m}\lambda v}\\
&=\sum_{|\alpha|\le m}\mu^{\alpha_d-2m}D^\alpha \tilde{f}_\alpha
\end{align*}
with the conormal derivative boundary condition on $\{x_d=0\}$, where
$$
\tilde{A}^{\alpha\beta}(t)=A^{\alpha\beta}(\mu^{-2m}t),\quad
\tilde{f}_\alpha(t,x^\prime,x_d)=f_\alpha(\mu^{-2m}t, x^\prime, \mu^{-1}x_d).
$$
We leave the term $(-1)^mD^{\hat{\alpha}}(\tilde{A}^{\hat{\alpha}\hat{\alpha}}D^{\hat{\alpha}}v)$ on the left-hand side and move all the other spatial derivatives to the right-hand side and add $(-1)^m\cD_{d-1}^mv$ to both sides of the system,
\begin{equation*}
v_t+(-1)^m \big(D^{\hat{\alpha}} (A^{\hat\alpha\hat{\alpha}}D^{\hat{\alpha}} v)+\cD_{d-1}^mv\big)+\mu^{-2m}\lambda v=\sum_{|\alpha|\le m}D^\alpha \hat{f}_\alpha+(-1)^m\cD_{d-1}^mv,
\end{equation*}
where $\hat{f}_\alpha=\mu^{\alpha_d-2m}\tilde{f}_\alpha$ for $|\alpha|<m$,
$$\hat{f}_\alpha=\mu^{\alpha_d-2m}\tilde{f}_\alpha+ \sum_{{|\beta|=m}}(-1)^{m+1}\mu^{\alpha_d+\beta_d-2m} \tilde{A}^{\alpha\beta}(t)D^\beta v $$ {for} $|\alpha|=m$ {but} $\alpha\neq \hat{\alpha}$, and
$$\hat{f}_{\hat{\alpha}}=\mu^{-m}\tilde{f}_{\hat{\alpha}}+ \sum_{^{\beta\neq\hat{\alpha}}_{|\beta|=m}}(-1)^{m+1}\mu^{\beta_d-m} \tilde{A}^{\hat{\alpha}\beta}(t)D^\beta v.$$
Then we implement Theorem \ref{L_p special} to get
\begin{align*}
&\mu^{-m}\lambda^{\frac{1}{2}}\|D^m v\|_{L_p(\o_{\mu^{2m}T}^+)}\\
&\le N\sum_{|\alpha|\le m}({\mu^{{-2m}}\lambda})^{\frac{|\alpha|}{2m}}\|\hat{f}_\alpha\|_{L_p(\o_{\mu^{2m}T}^+)}+N\mu^{-m}\lambda^\.5\|D_{x^\prime}D^{m-1}v\|_{L_p(\o_{\mu^{2m}T}^+)}\nonumber\\
&\le N\sum_{|\alpha|\le m}(\mu^{{-2m}}\lambda)^{\frac{|\alpha|}{2m}}\mu^{\alpha_d-2m}\|\tilde{f}_\alpha\|_{L_p(\o_{\mu^{2m}T}^+)}\\
&\quad+N\mu^{-m}\lambda^{\.5}(\sum_{^{\alpha\neq\hat{\alpha}}_{|\alpha|=m}}\sum_{|\beta|=m}\mu^{\alpha_d+\beta_d-2m}\|D^\beta v\|_{L_p(\o_{\mu^{2m}T}^+)}\\
&\quad+\sum_{^{\beta\neq\hat{\alpha}}_{|\beta|=m}}\mu^{\beta_d-m}\|D^\beta v\|_{L_p(\o_{\mu^{2m}T}^+)})+N\mu^{-m}\lambda^\.5\|D_{x^\prime}D^{m-1}v\|_{L_p(\o_{\mu^{2m}T}^+)}.
\end{align*}
Let $\mu$ be sufficiently large such that
$$
N\sum_{^{\alpha\neq\hat{\alpha}}_{|\alpha|=m}}
\sum_{|\beta|=m}\mu^{\alpha_d+\beta_d-2m}+
\sum_{^{\beta\neq\hat{\alpha}}_{|\beta|=m}}\mu^{\beta_d-m}\le 1/2.
$$
Then we fix this $\mu$ and obtain
\begin{align*}
\mu^{-m}\lambda^{\frac{1}{2}}\|D^m v\|_{L_p(\o_{\mu^{2m}T}^+)}\le& N\sum_{|\alpha|\le m}(\mu^{-2m}\lambda)^{\frac{|\alpha|}{2m}}\mu^{\alpha_d-2m}\|\tilde{f}_\alpha\|_{L_p(\o_{\mu^{2m}T}^+)}\\
&+N\mu^{-m}\lambda^\.5\|D_{x^\prime}D^{m-1}v\|_{L_p(\o_{\mu^{2m}T}^+)}.
\end{align*}
After returning to $u$ and $f_\alpha$, we prove the lemma.
 \end{proof}
We localize Lemma \ref{nor and tan} to get Lemma \ref{5.2} following the proof of Lemma 1 in \cite{DK10}.
\begin{lemma}\label{5.2}
Let $0<r<R<\infty$ and $p\in[2,\infty)$. Assume $u\in \h_p^m(Q_R^+)$ and
\begin{equation*}
u_t+(-1)^m \l_0 u=0 \quad \text{in} \quad Q_R^+
\end{equation*}
with the conormal derivative boundary condition on $\{x_d=0\}\cap Q_R$. Then there exists a constant $N=N(d,m,n,\delta,r,R,p,K)$ such that
\begin{equation}
\|D^m u\|_{L_p(Q_r^+)}\le N(\|D_{x^\prime}D^{m-1} u \|_{L_p(Q_R^+)}+\|u\|_{L_p(Q_R^+)})\label{lp boot}.
\end{equation}
\end{lemma}
Next, we state a parabolic type Sobolev embedding theorem.
\begin{lemma}\label{sobolev}
Let $r\in(0,\infty)$ and $1\le q\le p<\infty$. Assume that
\begin{equation*}
\frac{1}{q}-\frac{1}{p}\le\frac{1}{d+2m}.
\end{equation*}
Let $\zeta\in C_0^\infty$ be such $\zeta=1$ in $Q_r^+$. Then for any function $u$ such that $u\zeta\in \h_q^m(\o_\infty^+)$, we have $D^{m-1}u\in L_p(Q_r^+)$ and
\begin{equation*}
\|D^{m-1}u\|_{L_p(Q_r^+)}\le N\|u\zeta\|_{\h_q^m(\o_\infty^+)},
\end{equation*}
where $N=N(d,r,p,q)$.
\end{lemma}

{In the following lemma, we obtain  {a} H\"older {estimate} of $D^{m-1}u$.}
\begin{lemma}\label{regularity lemma}
Assume that $u\in C_{loc}^\infty(\overline{{\o_\infty^+}})$ and satisfies
\begin{equation*}
u_t+(-1)^m\l_0 u=0 \quad \text{in}\quad Q_2^+
\end{equation*}
with the conormal derivative boundary condition on $\{x_d=0\}\cap Q_2$. Then for any $\gamma\in (0,1)$ there exists a constant $N=N(d,m,n,\delta,\gamma,K)$ such that
\begin{equation*}
\|D^{m-1}u\|_{\frac{\gamma}{2m}, \gamma, Q_1^+}\le N \|u\|_{L_2(Q_2^+)}.
\end{equation*}
\end{lemma}
\begin{proof}
Due to Lemma \ref{L2 local} and the definition of $\|\cdot\|_{\h_p^m}$, we have
\begin{equation*}
\|u\|_{\h_2^m(Q_{r_1}^+)}\le N(d,r_1)\|u\|_{L_2(Q_2^+)},
\end{equation*}
where $1<r_1<2$.
From Lemma \ref{sobolev}, we know that there is a $p_1$ satisfying $\frac{1}{2}-\frac{1}{p_1}\le \frac{1}{d+2m}$ such that
\begin{equation*}
\|D^{m-1}u\|_{L_{p_1}(Q_{r_1^\prime}^+)}\le N(d,p_1) \|u\|_{\h_2^m(Q_{r_1}^+)}\le N(d,r_1,p_1)\|u\|_{L_2(Q_2^+)},
\end{equation*}
where $1<r_1^\prime<r_1$.
Since $D_{x^\prime} u$ satisfies the same system and boundary condition as $u$, with slight modification of the argument above and {Lemma \ref{L2 local}} we can show that
\begin{equation*}
\|D_{x^\prime}D^{m-1}u\|_{L_{p_1}(Q_{r_1^\prime}^+)}\le  N(d,r_1,p_1)\|u\|_{L_2(Q_2^+)}.
\end{equation*}
From \eqref{lp boot}, choosing $1<r_2<r_1^\prime$ so that
\begin{equation*}
\|D^m u\|_{L_{p_1}(Q_{r_2}^+)}\le N (\|D_{x^\prime}D^{m-1}u\|_{L_{p_1}(Q^+_{r^\prime_1})}+\|u\|_{L_{p_1}(Q_{r^\prime_1}^+)})\le N\|u\|_{L_2(Q_2^+)},
\end{equation*}
which implies
\begin{equation*}
\|u\|_{\h_{p_1}^m(Q_{r_2}^+)}\le N \|u\|_{L_2(Q_2^+)}.
\end{equation*}
By induction, we can choose an increasing sequence $p_1,\,p_2,\ldots$, such that
\begin{equation*}
\frac{1}{p_{i}}-\frac{1}{p_{i+1}}\le \frac{1}{d+2m},
\end{equation*}
and a sequence of decreasing domains $Q_{r_1}^+\supset Q_{r_2}^+\supset \cdots$, such that
\begin{equation*}
\|u\|_{\h_{p_i}^m(Q_{r_{2i}}^+)}\le N\|u\|_{L_2(Q_2^+)}.
\end{equation*}
It is obvious that for any $\gamma\in(0,1)$, in finite steps we can always take $p_n= (d+2m)/(1-\gamma)$ and $Q_{r_{2n}}^+\supset Q_1^+$. Finally applying the Sobolev embedding theorem again, we prove the lemma.
\end{proof}

\begin{corollary}\label{mean basic}
Let $\lambda\ge 0$, $\gamma\in(0,1)$, and $X_0=(t_0,x_0^\prime,0)$, where $t_0\in \bR$ and $x_0^\prime\in \bR^{d-1}$. Assume that $u\in C_{loc}^\infty(\overline{{\o_\infty^+}})$ satisfies
\begin{equation*}
u_t+(-1)^m \l_0 u+\lambda u=0
\end{equation*}
in $Q_4^+$ with the conormal derivative boundary condition on $\{x_d=0\}\cap Q_4$. Then there exists a constant $N=N(d,m,n,\delta,\gamma{,K})$ such that
\begin{equation}
[D_{x^\prime}D^{m-1} u]_{\frac{\gamma}{2m},\gamma, Q_1^+(X_0)}+\lambda^{\frac{1}{2}}[u]_{\frac{\gamma}{2m},\gamma, Q_1^+(X_0)} \le N\sum_{k=0}^m\lambda^{\frac{1}{2}-\frac{k}{2m}}\|D^k u\|_{L_2(Q_4^+(X_0))}.\label{tan basic}
\end{equation}

\end{corollary}
\begin{proof}
By using a translation in $t$ and $x^\prime$, we may assume that $X_0=(0,0)$. First let $\lambda =0$. The inequality in the corollary  {becomes}
\begin{equation*}
[D_{x^\prime}D^{m-1}u]_{\frac{\gamma}{2m},\gamma, Q_1^+}\le N\|D^m u\|_{L_2(Q_4^+)}.
\end{equation*}
We differentiate the system with respect to $x^\prime$ and apply Lemma \ref{regularity lemma} to $D_{x^\prime}u$ to obtain
\begin{equation*}
[D_{x^\prime}D^{m-1}u]_{\frac{\gamma}{2m},\gamma, Q_1^+}\le N\|D_{x^\prime} u\|_{L_2(Q_4^+)}.
\end{equation*}
Then we apply the same method as in the proof of Lemma \ref{tech 3} by considering $u-P$.

 In order to handle the case $\lambda>0$, we implement an argument originally due to S. Agmon.  Specifically, let
\begin{equation*}
\zeta (y)=\cos(\lambda^{\frac{1}{2m}}y)+\sin(\lambda^{\frac{1}{2m}}y).
\end{equation*}
Note that
\begin{equation*}
(-1)^m D_y^{2m}\zeta(y)=\lambda \zeta(y),\quad \zeta(0)=1,\quad |D^m\zeta(0)|=\lambda^\.5.
\end{equation*}
Denote $(t,z)=(t,y,x)$ to be a point in $\bR^{d+2}$ and set
\begin{equation*}
\hat{u}(t,z)=u(t,x)\zeta(y),\quad \hat{Q}_r^+=(-r^{2m},0)\times\{z:|z|<r,z_{d+1}>0,z\in\bR^{d+1}\}.
\end{equation*}
Obviously, $\hat{u}$ satisfies
\begin{equation*}
\hat{u}_t+(-1)^m\l_0 \hat{u}+(-1)^m D_y^{2m}\hat{u}=0
\end{equation*}
in $\hat{Q}_4^+$ with the conormal derivative boundary condition on $\{z_{d+1}=0\}\cap \hat{Q}_4$. Note that although our new operator above is not strongly elliptic, it satisfies \eqref{eq3.06}. By Remark \ref{rem4.2}, Theorem \ref{L_2} and thus Lemmas \ref{nor and tan}-\ref{regularity lemma} still hold. Upon applying the lemma with $\lambda=0$ to $\hat{u}$, we find
\begin{equation}
[D_{z^\prime}D^{m-1}\hat{u}]_{\frac{\gamma}{2m},\gamma, \hat{Q}_1^+}\le N\|D^m \hat{u}\|_{L_2(\hat{Q}_4^+)},\label{sa estimate}
\end{equation}
{where $z'=(z_1,\ldots,z_d)$}. Since, for example
\begin{equation*}
\lambda^\.5[u]_{\frac{\gamma}{2m},\gamma, Q_1^+}\le [D^m_y \hat{u}]_{\frac{\gamma}{2m},\gamma,\hat{Q}_1^+},
\end{equation*}
we only need to bound the right-hand side of \eqref{sa estimate} by the right-hand side of \eqref{tan basic}.
This can be done easily, since $D^m \hat{u}$ is a linear combination of terms like
\begin{equation*}
\lambda^{\.5-\frac{k}{2m}}\cos(\lambda^{\frac{1}{2m}}y)D_x^ku(t,x),\quad \lambda^{\.5-\frac{k}{2m}}\sin(\lambda^{\frac{1}{2m}}y)D_x^ku(t,x),\quad k=0,\ldots,m.
\end{equation*}
{The corollary is proved.}
\end{proof}

In the following lemma we obtain a mean oscillation estimate of $D_{x^\prime}D^{m-1}u$ for homogeneous systems.
\begin{lemma}\label{mean tan}
Let $r\in(0,\infty),\kappa\in[64,\infty)$, $\gamma\in (0,1)$, and $X_0=(t_0,x_0)\in \overline{\o_\infty^+}$. Assume that $u\in C_{loc}^\infty(\overline{{\o_\infty^+}})$ satisfies
\begin{equation*}
u_t+(-1)^m \l_0 u+\lambda u=0
\end{equation*}
in $Q_{\kappa r}^+(X_0)$ with the conormal derivative boundary condition on $\{x_d=0\}\cap Q_{\kappa r}(X_0)$. Then
\begin{align*}
(|D_{x^\prime}D^{m-1}u-(D_{x^\prime}D^{m-1}u)_{Q_r^+(X_0)}|)_{Q_r^+(X_0)}+\lambda^{\frac{1}{2}}(|u-(u)_{Q_r^+(X_0)}|)_{Q_r^+(X_0)}\nonumber\\
\le N\kappa^{-\gamma}\sum_{k=0}^m \lambda^{\frac{1}{2}-\frac{k}{2m}}(|D^k u|^2)_{Q_{\kappa r}^+(X_0)}^\frac{1}{2},
\end{align*}
where $N=N(d,n,m,\delta,\gamma{,K})$.
\end{lemma}
\begin{proof}
Using Corollary \ref{mean basic}, the proof is exactly the same as Lemma \ref{mean special1} and thus is omitted.
\end{proof}
{In the following proposition, we obtain a mean oscillation estimate {of $D_{x^\prime}D^{m-1}u$} for systems with simple coefficients.}
\begin{proposition}\label{fund proposition}
Let $r\in (0,\infty),\kappa\in [128,\infty),\lambda> 0, \gamma\in(0,1)$, and $X_0=(t_0,x_0)\in \overline{\o_\infty^+}$. Assume that $u\in C_{loc}^\infty (\overline{\o_\infty^+})$ satisfies
\begin{equation*}
u_t+(-1)^m \l_0 u+\lambda u=\sum_{|\alpha|\le m}D^\alpha f_\alpha
\end{equation*}
in $Q^+_{\kappa r}(X_0)$ with the conormal derivative boundary condition on $\{x_d=0\}$, where $f_\alpha \in L_{2,loc}(\o_\infty^+)$, $|\alpha|\le m$. Then we have
\begin{align*}
(|D_{x^\prime}D^{m-1} u-(D_{x^\prime}D^{m-1} u)_{Q_r^+(X_0)}|)_{Q_r^+(X_0)}+\lambda^{\frac{1}{2}}(|u-(u)_{Q_r^+(X_0)}|)_{Q_r^+(X_0)}\\
\le N\kappa^{-\gamma}\sum_{k=0}^m\lambda^{\frac{1}{2}-\frac{k}{2m}}(|D^k u|^2)_{Q_{\kappa r}^+(X_0)}^\frac{1}{2}+N\kappa^{m+\frac{d}{2}}\sum_{|\alpha|\le m}\lambda^{\frac{|\alpha|}{2m}-\.5}(|f_\alpha|^2)^{\frac{1}{2}}_{Q_{\kappa r}^+(X_0)},
\end{align*}
where $N=N(d,m,n,\delta,\gamma{,K})$.
\end{proposition}
\begin{proof}
Let $\zeta$ be the function defined at the beginning of the proof of Theorem \ref{mean half space}.
By Theorem \ref{L_2}, there exists a unique solution $w\in \h^m_2(\o_\infty^+)$ of
\begin{equation*}
w_t+(-1)^m \l_0 w+\lambda w=\sum_{|\alpha|\le m}D^\alpha (\zeta f_\alpha)
\end{equation*}
in $\o_\infty^+$ with the conormal derivative boundary condition on $\{x_d=0\}$. Moreover,
\begin{equation*}
\sum_{|\alpha|\le m} \lambda^{1-\frac{|\alpha|}{2m}}\|D^\alpha w\|_{L_2(\o_{t_0}^+)}\le N\sum_{|\alpha|\le m}\lambda^{\frac{|\alpha|}{2m}}\|\zeta f_\alpha\|_{L_2(\o_{t_0}^+)},
\end{equation*}
from which, we get
\begin{align*}
(|D^m w|^2)^{\frac{1}{2}}_{Q_r^+(X_0)}
+\lambda^{\frac{1}{2}}(|w|^2)_{Q_r^+(X_0)}^{\frac{1}{2}}&\le N\kappa^{m+\frac{d}{2}}\sum_{|\alpha|\le m}\lambda^{\frac{|\alpha|}{2m}-\.5}(|f_\alpha|^2)_{Q_{\kappa r}^+(X_0)}^\.5,\\
\sum_{k=0}^m\lambda^{\.5-\frac{k}{2m}}(|D^k w|^2)^\.5_{Q_{\kappa r}^+(X_0)}&\le N\sum_{|\alpha|\le m}\lambda^{\frac{|\alpha|}{2m}-\.5}(|f_\alpha|^2)^\.5_{Q_{\kappa r}^+(X_0)}.
\end{align*}

By a mollification argument as in Lemma \ref{mean half space}, we can assume $w$ is smooth. {Let $v=u-w$, which is smooth as well and satisfies }
\begin{equation*}
v_t+(-1)^m\l_0 v+\lambda v=\sum_{|\alpha|\le m}D^\alpha ((1-\zeta)f_\alpha)
\end{equation*}
in $\o_\infty^+$ with the conormal derivative boundary condition on $\{x_d=0\}$.  Hence in $Q_{\kappa r/2}^+(X_0)$
\begin{equation*}
v_t+(-1)^m\l_0 v+\lambda v=0.
\end{equation*}
Applying Lemma \ref{mean tan} to $v$,
\begin{align*}
(|D_{x^\prime}D^{m-1}v-(D_{x^\prime}D^{m-1}v)_{Q_r^+(X_0)}|)_{Q_r^+(X_0)}+\lambda^{\frac{1}{2}}(|v-(v)_{Q_r^+(X_0)}|)_{Q_r^+(X_0)}\\
\le N\kappa^{-\gamma}\sum_{k=0}^m \lambda^{\frac{1}{2}-\frac{k}{2m}}(|D^k v|^2)_{Q_{\kappa r/2}^+(X_0)}^{\frac{1}{2}}.
\end{align*}
With all the preparations above and following the proof of Theorem \ref{mean half space}, by the triangle inequality and the Cauchy--Schwarz inequality we have
\begin{align*}
&(|D_{x^\prime}D^{m-1} u-(D_{x^\prime}D^{m-1} u)_{Q_r^+{(X_0)}}|)_{Q_r^+(X_0)}+\lambda^{\frac{1}{2}}(|u-(u)_{Q_r^+(X_0)}|)_{Q_r^+(X_0)}\\
&\le N\bigg((|D_{x^\prime}D^{m-1} v-(D_{x^\prime}D^{m-1} v)_{Q_r^+{(X_0)}}|)_{Q_r^+(X_0)}\\
&\quad +\lambda^{\frac{1}{2}}(|v-(v)_{Q_r^+(X_0)}|)_{Q_r^+(X_0)}+(|D^m w|^2)^{\frac{1}{2}}_{Q_r^+(X_0)}+\lambda^{\frac{1}{2}}(|w|^2)_{Q_r^+(X_0)}^{\frac{1}{2}}\bigg)\\
&\le N\kappa^{-\gamma}\sum_{k=0}^m \lambda^{\frac{1}{2}-\frac{k}{2m}}(|D^k v|^2)_{Q_{\kappa r}^+(X_0)}^{\frac{1}{2}}+ N\kappa^{m+\frac{d}{2}}\sum_{|\alpha|\le m}\lambda^{\frac{|\alpha|}{2m}-\.5}(|f_\alpha|^2)_{Q_{\kappa r}^+(X_0)}^\.5\\
&\le  N\kappa^{-\gamma}\sum_{k=0}^m\lambda^{\frac{1}{2}-\frac{k}{2m}}(|D^k u|^2)_{Q_{\kappa r}^+(X_0)}^\frac{1}{2}+N\kappa^{m+\frac{d}{2}}\sum_{|\alpha|\le m}\lambda^{\frac{|\alpha|}{2m}-\.5}  (|f_\alpha|^2)^{\frac{1}{2}}_{Q_{\kappa r}^+(X_0)}.
\end{align*}
Therefore, we prove the proposition.
\end{proof}
{Next, we consider the case that $A^{\alpha\beta}$ are functions of both $x$ and $t$ and use the argument of freezing the coefficients to obtain:}
\begin{lemma}\label{important tan lemma}
Let $\l $ be the operator in Theorem \ref{theorem tan}. Suppose the lower-order coefficients of $\l$ are all zero. Let $\xi,\nu\in(1,\infty)$ satisfying $1/\xi+1/\nu=1$, {$\gamma\in(0,1)$,} and $\lambda,R\in (0,\infty)$. Assume $u\in C_0^\infty(\overline{\o^+_\infty})$ vanishing outside $Q_R^+(X_1)$, where $X_1\in \overline{\o_\infty^+}$, and
\begin{equation*}
u_t+(-1)^m\l u+\lambda u=\sum_{|\alpha|\le m}D^\alpha f_\alpha
\end{equation*}
with the conormal derivative boundary condition on $\{x_d=0\}$, where $f_\alpha\in L_{2,loc}(\o_\infty^+)$. Then there exists a constant $N=N(d,m,n,\delta ,\xi, K,\gamma)$ such that for any $r\in (0,\infty),\kappa\ge 128$, and $X_0\in \overline{\o_\infty^+}$, we have
\begin{align*}
&(|D_{x^\prime}D^{m-1}u-(D_{x^\prime}D^{m-1}u)_{Q_r^+(X_0)}|)_{Q_r^+(X_0)}+\lambda^\.5(|u-(u)_{Q_r^+(X_0)}|)_{Q_r^+(X_0)}\\
&\le N \kappa^{m+\frac{d}{2}}\big(\sum_{|\alpha|\le m}\lambda^{\frac{|\alpha|}{2m}-\.5}(|f_\alpha|^2)^\.5 _{Q_{\kappa r}^+(X_0)}+(A_R^\#)^{\frac{1}{2\nu}}(|D^m u|^{2\xi})^{\frac{1}{2\xi}}_{Q_{\kappa r}^+(X_0)}    \big)\\
&\quad+ N\kappa^{-\gamma}\sum_{k=0}^m\lambda^{\.5-\frac{k}{2m}}(|D^k u|^2)^\.5_{Q_{\kappa r}^+(X_0)}.\\
\end{align*}
\end{lemma}
\begin{proof}
Fix a $y\in \overline{\bR^{d}_+}$ and set $$\l_y u=\sum_{|\alpha|=|\beta|=m}D^\alpha (A^{\alpha\beta} (t,y)D^\beta u(t,x)).$$ Then we have
\begin{equation*}
u_t+(-1)^m\l_y u+\lambda u=\sum_{|\alpha|\le m}D^\alpha \tilde{f}_\alpha,
\end{equation*}
where

\begin{align*}
&\tilde{f}_\alpha=f_\alpha+(-1)^m\sum_{|\beta|=m}(A^{\alpha\beta}(t,y)-A^{\alpha\beta}(t,x))D^\beta u\quad \text{when}\quad |\alpha|=m,\\
&\tilde{f}_\alpha=f_\alpha\quad \text{otherwise}.
\end{align*}
It follows from Proposition \ref{fund proposition} that
\begin{align}
&(|D_{x^\prime}D^{m-1} u-(D_{x^\prime}D^{m-1} u)_{Q_r^+{(X_0)}}|)_{Q_r^+(X_0)}+\lambda^{\frac{1}{2}}(|u-(u)_{Q_r^+(X_0)}|)_{Q_r^+(X_0)}\nonumber\\
&\le N\kappa^{-\gamma}\sum_{k=0}^m\lambda^{\frac{1}{2}-\frac{k}{2m}}(|D^k u|^2)_{Q_{\kappa r}^+(X_0)}^\frac{1}{2}+N\kappa^{m+\frac{d}{2}}\sum_{|\alpha|\le m}\lambda^{\frac{|\alpha|}{2m}-\frac{1}{2}}(|\tilde{f}_\alpha|^2)^{\frac{1}{2}}_{Q_{\kappa r}^+(X_0)}\label{important 1}.
\end{align}
Note that
\begin{equation}
\int_{Q_{\kappa r}^+(X_0)}|\tilde{f}_\alpha|^2\,dx\,dt\le N\int_{Q_{\kappa r}^+(X_0)}|f_\alpha|^2\,dxdt+NI_y,\label{important 2}
\end{equation}
where, for $|\alpha|=m$,
\begin{equation*}
I_y=\int_{Q_{\kappa r}^+(X_0)}|(A^{\alpha\beta}(t,y)-A^{\alpha\beta}(t,x))D^\beta u|^2\,dx\,dt.
\end{equation*}
Denote $B^+$ to be $B^+_{\kappa r}(x_0)$ if $\kappa r<R$, or to be $B_R^+(x_1)$ otherwise; denote  $Q^+$ in the same fashion. Now we take the average of $I_y$ with respect to $y$ in $B^+$. Since $u$ vanishes outside $Q_R^+(X_1)$, by H\"{o}lder's  inequality we get
\begin{align*}
&\dashint_{B^+}I_y\, dy=\dashint_{B^+}\int_{Q^+_{\kappa r}(X_0)\cap Q_R^+(X_1)}|(A^{\alpha\beta}(t,y)-A^{\alpha\beta}(t,x))D^\beta u|^2\,dx\,dt\,dy\\
&\le \dashint_{B^+}\big(\int_{Q^+}|A^{\alpha\beta}(t,y)-A^{\alpha\beta}(t,x)|^{2\nu}
\,dx\,dt\big)^{\frac{1}{\nu}}\,dy \\
&\quad\cdot \big(\int_{Q^+_{\kappa r}(X_0)\cap Q_R^+(X_1)}|D^m u|^{2\xi}\,dx\,dt\big)^{\frac{1}{\xi}},
\end{align*}
where, by the boundedness of $A^{\alpha\beta}$, H\"older's inequality as well as the definition of $\text{osc}_x$ and $A^\#_R$, the integral over $B^+$ in the last term above is bounded by a constant times
\begin{align*}
&\dashint_{B^+}(\int_{Q^+}|A^{\alpha\beta}(t,y)-A^{\alpha\beta}(t,x)|^{2\nu}\,dx\,dt)^{\frac{1}{\nu}}\,dy\\
 &\le N(\dashint_{B^+}\int_{Q^+}|A^{\alpha\beta}(t,y)-A^{\alpha\beta}(t,x)|\,dx\,dt\,dy)^{\frac{1}{\nu}}\\
&\le N (|Q^+|\text{osc}_x(A^{\alpha\beta},Q^+))^{\frac{1}{\nu}}\le N((\kappa r)^{2m+d}A^\#_R)^{\frac{1}{\nu}}.
\end{align*}
This together with \eqref{important 1} and \eqref{important 2} completes the proof of the lemma.
\end{proof}
\section{Proof of Theorem \ref{theorem tan}}
\begin{proof}[Proof of Theorem \ref{theorem tan}]
First note that the interior estimates are obtained in Theorem 1 {of} \cite{DK10}. With the standard arguments of partition of the unity and flattening the boundary, it suffices to consider $\Omega=\bR_+^d$.

We may assume that all the lower-order coefficients are zero. {Indeed, if we got the a priori estimate without lower-order terms, for general systems, we can move lower-order terms to the right-hand side and apply the interpolation inequality. Taking $\lambda$ large enough, we then obtain the estimate for general systems.}

Case 1: $p\in(2,\infty)$. First we suppose $T=\infty$ and $u\in C^{\infty}(\overline{\o_\infty^+})$ and vanish{es} on $\o_\infty^+ \backslash \tilde{Q}_{R_0}^+(X_1)$ for some $X_1 \in \overline{\o_\infty^+} $, where $ \tilde{Q}_{R_0}^+(X_1)=\{(\mu^{-2m}t, x^\prime, \mu^{-1}x_d): (t,x)\in Q_{R_0}^+(X_1)\}$ and $\mu\ge1$ is a parameter which will be determined later. Then it follows $\tilde{Q}_{R_0}^+(X_1)\subset Q_{R_0}^+(\tilde{X}_1)$, where $\tilde{X_1}=(\mu^{-2m}t_1, x_1^\prime, \mu^{-1}{x_1}_d)$.  Choose $\xi>1$ such that $2\xi<p$ and {fix $\gamma\in (0,1)$}. Under these assumptions, from Lemma \ref{important tan lemma} we easily deduce
\begin{align*}
(D_{x^\prime}D^{m-1}u)^\# (X_0)+\lambda^\.5 u^\#(X_0)\le N \kappa^{-\gamma}\sum_{k=0}^m \lambda^{\.5-\frac{k}{2m}}(M(D^k u)^2(X_0))^\.5\\
+N \kappa^{m+\frac{d}{2}} \bigg( \sum_{|\alpha|\le m}\lambda^{\frac{|\alpha|}{2m}-\.5}(M(f_\alpha)^2(X_0))^\.5 +\rho^{\frac{1}{2\nu}}(M(D^m u)^{2\xi}(X_0))^{\frac{1}{2\xi}}  \bigg)
\end{align*}
for any $\kappa\ge128$ and $X_0\in \overline{\o_\infty^+}$.
This, together with the Fefferman--Stein theorem and the Hardy--Littlewood maximal function theorem, yields
\begin{align}
&\|D_{x^\prime}D^{m-1} u\|_{L_p}+\lambda^\.5\|u\|_{L_p}\le  {N(\|(D_{x^\prime}D^{m-1} u)^\#\|_{L_p}+\lambda^\.5\|u^\#\|_{L_p})\nonumber}\\
&\le {N\kappa^{-\gamma}\sum_{k=0}^m \lambda^{\.5-\frac{k}{2m}}\|(M(D^k u)^2)^\.5\|_{L_p}\nonumber } \\
&\,\, +N \kappa^{m+\frac{d}{2}} \big( \sum_{|\alpha|\le m}\lambda^{\frac{|\alpha|}{2m}-\.5}\|(M(f_\alpha)^2)^\.5\|_{L_p} +\rho^{\frac{1}{2\nu}}\|(M(D^m u)^{2\xi})^{\frac{1}{2\xi}}\|_{L_p}\big)\nonumber\\
&\le N \kappa^{m+\frac{d}{2}}\sum_{|\alpha|\le m}\lambda^{\frac{|\alpha|}{2m}-\.5}\|f_\alpha\|_{L_p} +N(\kappa^{-\gamma}+\kappa^{m+\frac{d}{2}}\rho^{\frac{1}{2\nu}})\sum_{k=0}^m \lambda^{\.5-\frac{k}{2m}}\|D^k u\|_{L_p}\label{lp tan}
\end{align}
for any $\kappa\ge 128$, where $L_p=L_p(\o_\infty^+)$.

Now we use the arguments of freezing the coefficients and scaling to get the estimate of $D_d^m u$. As in Lemma  \ref{nor and tan}, let $$v(t,x^\prime, x_d)=u(\mu^{-2m}t, x^\prime, \mu^{-1} x_d),$$ $$\tilde{A}^{\alpha\beta}(t,x^\prime,x_d)=A^{\alpha\beta}(\mu^{-2m}t, x^\prime,\mu^{-1}x_d),$$  $$\tilde{f}_\alpha(t,x^\prime,x_d)=f_\alpha(\mu^{-2m}t, x^\prime, \mu^{-1}x_d).$$ Then $v$ satisfies
\begin{align*}
&v_t+(-1)^m D_d^m(\tilde{A}^{\hat{\alpha}\hat{\alpha}}D^m_d v)+(-1)^{m}\sum_{^{|\alpha|=|\beta|=m}_{(\alpha,\beta)\neq (\hat{\alpha},\hat{\alpha})}}\mu^{\alpha_d+\beta_d-2m}D^\alpha(\tilde{A}^{\alpha\beta}D^\beta v)\\&\quad
+\lambda\mu^{{-2m}}v=\sum_{|\alpha|\le m}\mu^{\alpha_d-2m}D^\alpha f_\alpha
\end{align*}
in $\o_\infty^+$ with the conormal derivative boundary condition on $\{x_d=0\}$.
We fix $(t,y)\in \overline{\o_\infty^+}$ and move all spacial derivatives to the right-hand side of the equation, then add $(-1)^m (\cD_{d-1}^m v+D_d^m (
{\tilde{A}^{\hat{\alpha}\hat{\alpha}}(t,y)}D_d^m v))$ to both sides of the equation so that
\begin{align*}
&v_t+(-1)^m( D_d^m(\tilde{A}^{\hat{\alpha}\hat{\alpha}}(t,y)D^m_dv)+\cD_{d-1}^m v)+\lambda\mu^{-2m}v\\
&=D^\alpha \hat{f}_\alpha+(-1)^m D^m_d\big((\tilde{A}^{\hat{\alpha}\hat{\alpha}}(t,y)-\tilde{A}^{\hat{\alpha}\hat{\alpha}}(t,x))D^m_dv\big)+(-1)^m\cD_{d-1}^m v,
\end{align*}
where $\hat{f}_\alpha=\mu^{\alpha_d-2m}\tilde{f}_\alpha$ for $|\alpha|<m$,
$$\hat{f}_\alpha=\mu^{\alpha_d-2m}\tilde{f}_\alpha+ \sum_{{|\beta|=m}}(-1)^{m+1}\mu^{\alpha_d+\beta_d-2m} \tilde{A}^{\alpha\beta}(t,x)D^\beta v$$ {for} $|\alpha|=m$ but $\alpha\neq \hat{\alpha}$, and
$$\hat{f}_{\hat{\alpha}}=\mu^{-m}\tilde{f}_{\hat{\alpha}}+ \sum_{^{\beta\neq\hat{\alpha}}_{|\beta|=m}}(-1)^{m+1}\mu^{\beta_d-m} \tilde{A}^{\hat{\alpha}\beta}(t,x)D^\beta v.$$
We follow the proof of Lemma \ref{important tan lemma}. From \eqref{mean oscillation special1}, we know that
\begin{align*}
&(|D_d^m v-(D_d^m v)_{Q_r^+(X_0)}|)_{Q_r^+(X_0)}\nonumber\\
&\le N{\kappa_1}^{-1}(|D_d^m v|^2)^{\frac{1}{2}}_{Q_{\kappa_1 r}^+(X_0)}+N\kappa_1^{m+\frac{d}{2}}\Big(\sum_{|\alpha|\le m}\lambda^{\frac{|\alpha|}{2m}-\frac{1}{2}}(|\hat{f}_\alpha|^2)^{\frac{1}{2}}_{Q^+_{\kappa_1 r}(X_0)}\\
&\quad+(|(\tilde{A}^{\hat{\alpha}\hat{\alpha}}(t,y)
-\tilde{A}^{\hat{\alpha}\hat{\alpha}}(t,x))D^m_dv|^2)^{\.5}_{Q^+_{\kappa_1 r}(X_0)}+(|D_{x^\prime}^{m}v|^2)^\.5_{Q^+_{\kappa_1 r}(X_0)}\Big),
\end{align*}
for any $X_0\in \overline{\o_\infty^+}$ and $\kappa_1\ge 64$. It is easy to check that $\tilde{A}^{\alpha\beta}$ satisfies the Assumption \ref{main assumption}$(2\mu^{2m+1}\rho)$ with the same $R_0$ as $A^{\alpha\beta}$ and the support of $v$ is contained in $Q_{R_0}^+({X_1})$. Therefore, applying the same argument {as} in Lemma \ref{important tan lemma}, we obtain
\begin{align*}
&(|D_d^m v-(D_d^m v)_{Q_r^+(X_0)}|)_{Q_r^+(X_0)}\nonumber\\
&\le N{\kappa_1}^{-1}(|D_d^m v|^2)^{\frac{1}{2}}_{Q_{\kappa_1 r}^+(X_0)}+N\kappa_1^{m+\frac{d}{2}}\sum_{|\alpha|\le m}\lambda^{\frac{|\alpha|}{2m}-\frac{1}{2}}(|\hat{f}_\alpha|^2)^{\frac{1}{2}}_{Q^+_{\kappa_1 r}(X_0)}\\
&+N\kappa_1^{m+\frac{d}{2}} (2\mu^{2m+1}\rho)^{\frac{1}{2\nu}}(|D_d^m v|^{2\xi})^{\frac{1}{2\xi}}_{Q_{\kappa_1 r}^+(X_0)}+N\kappa_1^{m+\frac{d}{2}}(|D_{x^\prime}^{m}v|^2)^\.5_{Q^+_{\kappa_1 r}(X_0)},
\end{align*}
where $1/\xi+1/\nu=1$. By the definition of $\hat{f}_\alpha$, the right-hand side of the inequality above can be bounded by
\begin{align*}
&N{\kappa_1}^{-1}(|D_d^m v|^2)^{\frac{1}{2}}_{Q_{\kappa_1 r}^+(X_0)}+N\kappa_1^{m+\frac{d}{2}} (2\mu^{2m+1}\rho)^{\frac{1}{2\nu}}(|D_d^m v|^{2\xi})^{\frac{1}{2\xi}}_{Q_{\kappa_1 r}^+(X_0)}\\
&{+}N\kappa_1^{m+\frac{d}{2}}\sum_{|\alpha|\le m}\lambda^{\frac{|\alpha|}{2m}-\.5}\mu^{\alpha_d-2m}(|\tilde{f}_\alpha|^2)^\.5_{Q_{\kappa_1 r}^+(X_0)}+N\kappa_1^{m+\frac{d}{2}}\mu^{-1}(|D_d^m v|^2)^\.5_{Q_{\kappa_1 r}^+(X_0)}\\
&+N\kappa_1^{m+\frac{d}{2}}(|D_{x^\prime}D^{m-1}v|^2)^\.5_{Q_{\kappa_1 r}^+(X_0)}
\end{align*}
provided that $\mu\ge 1$. Therefore, we obtain that, for any $X_0\in \overline{\mathcal O^+_\infty}$,
\begin{align*}
&(D_d^m v)^\#{(X_0)}\\
&\le N{\kappa_1}^{-1}(M(D_d^m v)^2(X_0))^\.5+N\kappa_1^{m+\frac{d}{2}} (2\mu^{2m+1}\rho)^{\frac{1}{2\nu}}(M(D_d^m v)^{2\xi}(X_0))^{\frac{1}{2\xi}}\\
&\quad+N\kappa_1^{m+\frac{d}{2}}\sum_{|\alpha|\le m}\lambda^{\frac{|\alpha|}{2m}-\.5}\mu^{\alpha_d-2m}(M(\tilde{f_\alpha})^2(X_0))^\.5\\
&\quad+N\kappa_1^{m+\frac{d}{2}}\mu^{-1}(M(D_d^m v)^2(X_0))^\.5+N\kappa_1^{m+\frac{d}{2}} (M(D_{x^\prime}D^{m-1}v)^2(X_0))^\.5.
\end{align*}
We denote $L_p=L_p(\o_\infty^+)$. By the Fefferman--Stein theorem, the Hardy--Littlewood maximal function theorem, and choosing $\xi>1, 2\xi<p$, we get
\begin{align*}
&\|D_d^m v\|_{L_p}\le N\|(D_d^m v)^\#\|_{L_p}\\
&\le N{\kappa_1}^{-1}\|(M(D_d^m v)^2)^\.5\|_{L_p}+N\kappa_1^{m+\frac{d}{2}} (2\mu^{2m+1}\rho)^{\frac{1}{2\nu}}\|(M(D_d^m v)^{2\xi})^{\frac{1}{2\xi}}\|_{L_p}\\
&\quad +N\kappa_1^{m+\frac{d}{2}}\sum_{|\alpha|\le m}\lambda^{\frac{|\alpha|}{2m}-\.5}\mu^{\alpha_d-2m}\|(M(\tilde{f_\alpha})^2)^\.5\|_{L_p}\\
&\quad +N\kappa_1^{m+\frac{d}{2}}\mu^{-1}\|(M(D_d^m v)^2)^\.5\|_{L_p}+N\kappa_1^{m+\frac{d}{2}}\| (M(D_{x^\prime}D^{m-1}v)^2(X_0))^\.5\|_{L_p}\\
&\le  N{\kappa_1}^{-1}\|D_d^mv\|_{L_p}+N\kappa_1^{m+\frac{d}{2}} (2\mu^{2m+1}\rho)^{\frac{1}{2\nu}}\|D_d^m v\|_{L_p}\\
&\quad +N\kappa_1^{m+\frac{d}{2}}\sum_{|\alpha|\le m}\lambda^{\frac{|\alpha|}{2m}-\.5}\mu^{\alpha_d-2m}\|\tilde{f}_\alpha\|_{L_p}+N\kappa_1^{m+\frac{d}{2}}\mu^{-1}\|D_d^m v\|_{L_p}\\
&\quad +N\kappa_1^{m+\frac{d}{2}}\| D_{x^\prime}D^{m-1}v\|_{L_p}.
\end{align*}
We first let $\kappa_1$ be sufficiently large and then $\mu$ be sufficiently large. Finally let $\rho$ be sufficiently small and we obtain
\begin{align*}
\|D_d^m v\|_{L_p}\le N\sum_{|\alpha|\le m}\lambda^{\frac{|\alpha|}{2m}-\.5}\|\tilde{f_\alpha}\|_{L_p}+N\| D_{x^\prime}D^{m-1}v\|_{L_p}.
\end{align*}
After changing back to $u$ and $f_\alpha$, we get
\begin{equation*}
\|D_d^m u\|_{L_p}\le N\sum_{|\alpha|\le m}\lambda^{\frac{|\alpha|}{2m}-\.5}\|{f_\alpha}\|_{L_p}+N\| D_{x^\prime}D^{m-1}u\|_{L_p}.
\end{equation*}
Combining the inequality above with \eqref{lp tan}, we know that
\begin{align*}
&\sum_{k=0}^m\lambda^{\frac{1}{2}-\frac{k}{2m}}\|D^k u\|_{L_p}\le N\|D^m u\|_{L_p}+N\lambda^{\.5}\|u\|_{L_p}\\
&\le N\|D_{x^\prime}D^{m-1} u\|_{L_p}+N\lambda^\.5\|u\|_{L_p}+N\|D_d^m u\|_{L_p}\\
&\le  N \kappa^{m+\frac{d}{2}}\sum_{|\alpha|\le m}\lambda^{\frac{|\alpha|}{2m}-\.5}\|f_\alpha\|_{L_p}+N(\kappa^{-\gamma}+\kappa^{m+\frac{d}{2}}\rho^{\frac{1}{2\nu}})\sum_{k=0}^m \lambda^{\.5-\frac{k}{2m}}\|D^k u\|_{L_p}.
\end{align*}
We take $\kappa$ sufficiently large and $\rho$ sufficiently small so that the terms involving $u$ on the right-hand side are absorbed in the left-hand side. In this way, we prove the desired estimate.

Due to the argument of partition of the unity, we can remove the assumption that $u$ vanish{es} on $\o_\infty^+ \backslash \tilde{Q}_{R_0}^+(X_1)$ for some $X_1 \in \overline{\o_\infty^+}$ by choosing $\lambda_0$ large enough. The extension to $T\in (-\infty,\infty]$ is by now standard, see \cite{Kry07}.

Case 2: $p\in(1,2)$. Since the system is in the divergence form, this case follows from the previous case by using the duality argument.

Case 3: $p=2$.  This case is classical.
\end{proof}

\section{Schauder estimates for systems on the half space}

In this section, we prove the Schauder estimates for \eqref{the system1} on the half space.  The following lemmas are useful in our proof. {The first one is well known and  the proof can be found in \cite{Gia93}.}
\begin{lemma}\label{tech}
Let $\Phi$ be a nonnegative, nondecreasing function on $(0,r_0]$ such that
\begin{equation*}
\Phi(\rho)\le A(\frac{\rho}{r})^a\Phi(r)+Br^b
\end{equation*}
for $0<\rho<r\le r_0$, where $0<b<a$ are fixed constants. Then
\begin{equation*}
\Phi(r)\le Nr^b(r_0^{-b}\Phi(r_0)+B)\quad \forall\, r \in (0,r_0)
\end{equation*}
with a constant $N=N(A,a,b)$.
\end{lemma}
 The following version of Campanato's theorem can be found in \cite{Gia93} and \cite{Sh96}.
\begin{lemma}\label{holder}
(i) Let $f\in L_2(Q_2)$ and $a\in(0,1]$. Assume that
\begin{equation}
\dashint_{Q_r(t_0,x_0)}|f-(f)_{Q_r(t_0,x_0)}|^2\,dx\,dt\le A^2r^{2a},\quad {\forall\,} (t_0,x_0)\in Q_1,\label{holder1}
\end{equation}
and any $0<r \le 1$. Then we have $f \in C^{\frac{a}{2m},a}(Q_1)$ and
\begin{equation*}
[f]_{\frac{a}{2m},a,Q_1}\le NA,
\end{equation*}
with $N=N(d)$.

(ii) Let $f\in L_2(Q_2^+)$ and $a\in(0,1]$. Assume that \eqref{holder1} holds for $r<{x_0}_d$. Moreover
\begin{equation*}
\dashint_{Q_r^+(t_1,x_1)}|f-(f)_{Q_r^+(t_1,x_1)}|^2\,dx\,dt\le A^2r^{2a},\quad {\forall\,(t_1,x_1) \in \{x_d=0\}\cap Q_1},
\end{equation*}
and any $0<r\le 1$. Then we have $f \in C^{\frac{a}{2m},a}(Q_1^+)$ and
\begin{equation*}
[f]_{\frac{a}{2m},a,Q^+_1}\le NA,
\end{equation*}
with $N=N(d)$.
\end{lemma}

\subsection{The estimate of $D_{x^\prime}D_d^{m-1}u$}
\subsubsection{Systems with coefficients depending only on $t$}
In this part, we consider
\begin{equation}
u_t+(-1)^m \l_0 u=0 \quad \text{in } \quad Q_{2R}^+\label{simple coefficients}
\end{equation}
with the conormal derivative boundary condition on $\{x_d=0\}\cap Q_{2R}$. First, we have the following mean oscillation estimate.

\begin{lemma}\label{fund schauder}
Assume that $u\in C_{loc}^\infty(\overline{{\o_\infty^+}})$ and satisfies \eqref{simple coefficients} with the  conormal derivative boundary condition on $\{x_d=0\}$. Then for any $0<r<R<\infty$ and $\gamma\in (0,1)$, there exists a constant $N=N(d,m,n,\delta,K,\gamma)$ such that
\begin{align}
&\int_{Q_{r}^+(X_0)}|D^{m-1}u-(D^{m-1}u)_{Q_r^+(X_0)}|^2\,dx\,dt\nonumber\\
&\le N(\frac{r}{R})^{2\gamma+2m+d}\int_{Q_{R}^+(X_0)}|D^{m-1}u-(D^{m-1}u)_{Q_R^+(X_0)}|^2\,dx\,dt\label{mix schauder}
\end{align}
where $X_0\in \{x_d=0\}\cap Q_R$.
\end{lemma}
\begin{proof}
By scaling and translation of the coordinates, without loss of generality, we can assume $R=1$ and $ X_0=(0,0)$.

From the bootstrap argument in Lemma \ref{regularity lemma} and the Sobolev embedding theorem,  for any $p>0$, there exists a $N=N(d,p)$ so that
\begin{equation*}
\|u\|_{\h_p^m(Q_1^+)}\le N\|u\|_{L_2(Q_2^+)}.
\end{equation*}
By the parabolic Sobolev embedding theorem,
\begin{equation*}
[D^{m-1} u]_{\frac{\gamma}{2m},\gamma, Q_1^+}\le N\|u\|_{L_2(Q_2^+)},
\end{equation*}
where $\gamma\in (0,1)$.
Therefore,
\begin{equation*}
\int_{Q_{r}^+}|D^{m-1}u-(D^{m-1}u)_{Q_r^+}|^2\,dx\,dt\le Nr^{2\gamma+2m+d}\|u\|^2_{L_2(Q_1^+)}
\end{equation*}
provided $r\le 1/2$. Let $v$ be as in Lemma \ref{tech 3} and notice that $v$ satisfies the same system and boundary condition as $u$. So the inequality holds for $v$ as well. Due to Lemma \ref{tech 3},
\begin{equation*}
\int_{Q_{r}^+}|D^{m-1}v-(D^{m-1}v)_{Q_r^+}|^2\,dx\,dt\le Nr^{2\gamma+2m+d}\|D^{m-1}v\|^2_{L_2(Q_1^+)}.
\end{equation*}
Clearly the inequality above also holds true for $r>1/2$. Hence for any $r<1$,
\begin{align*}
&\int_{Q_{r}^+}|D^{m-1}u-(D^{m-1}u)_{Q_r^+}|^2\,dx\,dt\\
&\le Nr^{2\gamma+2m+d}\int_{Q_{1}^+}|D^{m-1}u-(D^{m-1}u)_{Q_1^+}|^2\,dx\,dt,
\end{align*}
because  $$D^{m-1} u-D^{m-1}v=D^{m-1}p=(D^{m-1}u)_{Q_1^+}.$$
The lemma is proved.
\end{proof}

\begin{lemma}\label{tan holder}
Assume that $u\in C_{loc}^{\infty}(\overline{{\o_\infty^+}})$ satisfies
\begin{equation}
\label{non homo simple} u_t+(-1)^m \l_0 u=\sum_{|\alpha|\le m}D^{\alpha}f_{\alpha} \quad \text{in} \quad Q_{4R}^+
\end{equation}
with the conormal derivative boundary condition on $\{x_d=0\}\cap Q_{4R}$ and  $f_\alpha \in C^{a\ast}$ if $|\alpha| =m$ for some $a\in (0,1)$, $f_\alpha \in L_\infty$ if $|\alpha|< m$.
Then for any $0<r<R \le 1$, $\gamma\in(0,1)$, and $X_0\in \{x_d=0\}\cap Q_R$, there exists a constant $N=N(n,m,d,\delta,\gamma,K)$ such that
\begin{align*}
&\int_{Q_r^+(X_0)}|D_{x^\prime}D^{m-1}u-(D_{x^\prime}D^{m-1}u)_{Q_r^+(X_0)}|^2\,dx\,dt \\&\le N (\frac{r}{R})^{2\gamma+2m+d} \int_{Q_R^+(X_0)}|D_{x^\prime}D^{m-1}u-(D_{x^\prime}D^{m-1}u)_{Q_R^+(X_0)}|^2\,dx\,dt\\
&\quad+N\sum_{|\alpha|=m}[f_\alpha]_{a, Q_{4R}^+} ^{\ast2} R^{2a+2m+d} +N\sum_{|\alpha|<m} \|f_\alpha\|_{L_\infty(Q_{4R}^+)}^2R^{2+2m+d}.
\end{align*}
\end{lemma}

\begin{proof}
Let $\zeta\in C_0^\infty(\bR^{d+1})$ and
\begin{equation*}
\zeta=1\quad \text{in}\quad Q_{2R}, \quad \zeta=0\quad \text{outside} \quad (-(4R)^{2m},(4R)^{2m})\times B_{4R}.
\end{equation*}
For $T=-(4R)^{2m}$, consider the following system,
\begin{equation*}
w_t+(-1)^m\l_0 w=\sum_{|\alpha|\le m}D^\alpha (\zeta \tilde{f}_\alpha)
\end{equation*}
in $(-T,0)\times \bR^{d}_+$ with the conormal derivative boundary condition on $\{x_d=0\}$ and the zero initial condition on $\{-T\}\times\bR_+^{d}$, where $$\tilde{f}_\alpha(t,x)=f_\alpha(t,x)-f_\alpha(t,0)$$ if $|\alpha|=m$, and $\tilde{f}_\alpha=f_\alpha$ otherwise.
From Theorem \ref{L_2}, the above system has a unique solution $w\in \h_2^m((-T,0)\times \bR^d_+)$. {Indeed, we can consider the system which $\tilde{w}=e^{-\lambda t}w$ satisfies, where $\lambda>0$. It is easy to see that from Theorem \ref{L_2} we can solve for $\tilde{w}$, and then obtain $w$.}  By the mollification argument as in Lemma \ref{mean half space}, we may assume that $w$ is smooth.
{Let $v=u-w$, which satisfies}
\begin{equation*}
v_t+(-1)^m \l_0 v=\sum_{|\alpha|=m}D^\alpha f_\alpha(t,0) \quad \text{in}\quad Q_{2R}^+
\end{equation*}
with the conormal derivative boundary condition on $\{x_d=0\}$. We differentiate above system with respect to $x^\prime$ and let $\hat{v}=D_{x^\prime}v$. Then $\hat{v}$ satisfies  all the conditions in Lemma \ref{fund schauder} so that  \eqref{mix schauder} holds for $\hat{v}$. Therefore, we obtain
\begin{align}
&\int_{Q_r^+(X_0)}|D_{x^\prime}D^{m-1}v-(D_{x^\prime}D^{m-1}v)_{Q_r^+(X_0)}|^2\,dx\,dt\nonumber \\&\le N (\frac{r}{R})^{2\gamma+2m+d} \int_{Q_R^+(X_0)}|D_{x^\prime}D^{m-1}v-(D_{x^\prime}D^{m-1}v)_{Q_R^+(X_0)}|^2\,dx\,dt\label{v estimate 1}.
\end{align}
From the proof of Theorem \ref{L_2},
\begin{align}
&\delta\int_{Q_{4R}^+} |D^m w|^2 \,dx\,dt\nonumber \\
&\le N\sum_{|\alpha|=m}\int_{Q_{4R}^+}|(f_\alpha(t,x)-f_\alpha(t,0))D^\alpha w| \,dx\,dt+\sum_{|\alpha|<m} \int_{Q_{4R}^+}|f_\alpha D^\alpha w|\,dx\,dt\nonumber\\
&\le \varepsilon \|D^m w\|^2_{L_2(Q_{4R}^+)}+N(\varepsilon)\sum_{|\alpha|=m}\|f_\alpha(t,x)-f_\alpha(t,0)\|^2_{L_2(Q_{4R}^+)}\nonumber\\
&\quad+\sum_{|\alpha|<m}\int_{Q_{4R}^+} |f_\alpha D^\alpha w|\,dx\,dt.\label{w estimate2}
\end{align}
We take $\varepsilon$ sufficiently small so that the first term on the right-hand side is absorbed in the left-hand side. The second term is less than $$\sum_{|\alpha|=m}N[f_\alpha]_{a, Q_{4R}^+}^{\ast2}R^{2a+d+2m}.$$ The last term can be estimated in the following way. We choose a vector-valued polynomial as in Lemma \ref{tech 3} with respect to $w$. Let $h=w-P(x)$. {Because $P(x)$ is of degree $m-1$,  $h$ satisfies the same system and boundary condition as $w$}. Moreover $D^m h$=$D^m w$. Therefore,  \eqref{w estimate2} holds with $h$ in place of $w$ in the last term of the right-hand side. By Lemma \ref{tech 3} with $4R$ in place of $R$ and the Cauchy--Schwarz inequality,
\begin{align}
&\int_{Q_{4R}^+}|f_\alpha D^\alpha h |\,dx\,dt\le N\|f_\alpha\|_{L_\infty(Q_{4R}^+)}|Q_{4R}^+|^\.5\|D^\alpha h\|_{L_2(Q_{4R}^+)}\nonumber\\
&\le N \|f_\alpha\|_{L_\infty(Q_{4R}^+)}|Q_{4R}^+|^\.5 \big(R^{m-|\alpha|}\|D^m w\|_{L_2(Q_{4R}^+)}\nonumber\\
&\quad+\sum_{|\beta|\le m}R^{3m+d/2-|\alpha|-|\beta|}\|\tilde{f}_\beta\|_{L_\infty(Q_{4R}^+)}\big).\label{eq 7.37}
\end{align}
Since $R<1$ and for $|\beta|=m$
$$
\|\tilde{f}_\beta\|_{L_\infty(Q_{4R}^+)}\le N[f_\beta]_{a, Q_{4R}^+}^\ast R^a,
$$ by Young's inequality and the Cauchy--Schwarz inequality, one can bound {the right-hand side of} \eqref{eq 7.37} by
\begin{align}
&\varepsilon \|D^m w\|^2_{L_2(Q_{4R}^+)}\nonumber \\
&+N(\varepsilon)(\sum_{|\alpha|<m} \|f_\alpha\|^2_{L_\infty(Q_{4R}^+)}R^{2+2m+d}+\sum_{|\alpha|=m}[f_\alpha]_{a, Q_{4R}^+}^{\ast2}R^{2a+2m+d}).\label{low order}
\end{align}
Finally, choosing $\varepsilon$ sufficiently small and combining \eqref{v estimate 1}-\eqref{low order}, by the triangle inequality we immediately prove the lemma. \end{proof}

Thanks to Lemma \ref{tech}, after taking $\gamma>a$,  the following inequality holds for any $r\in (0,R)$
\begin{align*}
&\int_{Q_r^+(X_0)}|D_{x^\prime}D^{m-1}u-(D_{x^\prime}D^{m-1}u)_{Q_r^+(X_0)}|^2\,dx\,dt \\&\le N (\frac{r}{R})^{2a+2m+d} \int_{Q_R^+(X_0)}|D_{x^\prime}D^{m-1}u-(D_{x^\prime}D^{m-1}u)_{Q_R^+(X_0)}|^2\,dx\,dt\\
&\quad +N\sum_{|\alpha|=m}[f_\alpha]_{a, Q_{4R}^+} ^{\ast2} r^{2a+2m+d} +N\sum_{|\alpha|<m} \|f_\alpha\|_{L_\infty(Q_{4R}^+)}^2r^{2a+2m+d}.
\end{align*}
By Lemma \ref{holder} together with the corresponding interior estimate, we obtain
\begin{equation}
[D_{x^\prime}D^{m-1}u]_{\frac{a}{2m},a,Q_R^+}\le N (R^{-(2a+d+2m)}\int_{Q_{4R}^+}|D^m u|^2\,dx\,dt+F^2)^{\frac{1}{2}},\label{except m}
\end{equation}
where $F=\sum_{|\alpha| <m}\|f_\alpha\|_{L_\infty(Q_{4R}^+)}+\sum_{|\alpha|=m}[f]_{a, Q_{4R}^+}^\ast $.

\subsubsection{Variable coefficients depending on both  $x$ and $t$}In this part, we use the argument of freezing the coefficients  to deal with the case {when} $A^{\alpha\beta}$ {depend on} both $x$ and $t$. First, let us consider systems which only consists of highest order terms:
\begin{equation}
u_t+(-1)^m \l_hu=u_t+(-1)^m \sum_{|\alpha|=|\beta|=m}D^\alpha (A^{\alpha\beta}D^\beta u)=\sum_{|\alpha|\le m}D^\alpha f_\alpha\label{highest order}
\end{equation}
in $Q_{4R}^+$ and follow the steps in the simple coefficient case. Let $X_0\in Q_R^+$ and the equation above can be written as
\begin{equation*}
u_t+(-1)^m \l_{0x_0}u=\sum_{|\alpha|\le m}D^\alpha \tilde{f}_\alpha,
\end{equation*}
where
\begin{equation*}
\l_{0x_0}=\sum_{|\alpha|=|\beta|=m}D^{\alpha}(A^{\alpha\beta}(t,x_0)D^\beta),
\end{equation*}
and
$$
\tilde{f}_\alpha=f_\alpha+(-1)^m \sum_{|\beta|=m}(A^{\alpha\beta}(t,x_0)-A^{\alpha\beta}(t,x))D^\beta u
$$
when $|\alpha|=m$, and $\tilde{f}_\alpha=f_\alpha$ otherwise.
Following exactly the same argument as in Lemma \ref{tan holder} with $\tilde{f_\alpha}$ in place of $f_\alpha$, we can prove the following lemma corresponding to Lemma \ref{tan holder}.
\begin{lemma}\label{high tan}
Assume that $a\in(0,1)$, $u\in C_{loc}^{\infty}(\overline{{\o_\infty^+}})$ satisfies \eqref{highest order} with the conormal derivative boundary condition on $\{x_d=0\}\cap Q_{4R}$, $f_\alpha \in C^{a\ast}$ if $|\alpha| =m$, $f_\alpha \in L_\infty$ if $|\alpha|< m$, and $A^{\alpha\beta}\in C^{a\ast}$. Then for any $0<r<R\le 1, \gamma\in(0,1)$ and $X_0\in \overline{Q_R^+}$, there exists a constant $N=N(d,n,m,\delta,K,\gamma,\|A^{\alpha\beta}\|_a^\ast)$ such that
\begin{align*}
&\int_{Q_r^+(X_0)}|D_{x^\prime}D_{}^{m-1}u-(D_{x^\prime}D^{m-1}u)_{Q_r^+(X_0)}|^2\,dx\,dt\\
&\le N(\frac{r}{R})^{2\gamma+2m+d}\int_{Q_R^+(X_0)}|D_{x^\prime}D^{m-1}u-(D_{x^\prime}D^{m-1}u)_{Q_R^+(X_0)}|^2\,dx\,dt\\
&\quad+N(\sum_{|\alpha|=m}[f_\alpha]_{a, Q_{4R}^+(X_0)} ^{\ast2}+ \|D^m u\|^2_{L_\infty(Q_{4R}^+(X_0))})R^{2a+2m+d}
\\
&\quad+N\sum_{|\alpha|<m} \|f_\alpha\|_{L_\infty(Q_{4R}^+(X_0))}^2R^{2+2m+d}.
\end{align*}
\end{lemma}
By Lemmas \ref{tech} and \ref{holder}, we derive the following corollary by taking $\gamma>a$.
\begin{corollary}\label{cor tan}
Under the conditions in Lemma \ref{high tan}, we have
\begin{equation*}
[D_{x^\prime}D^{m-1}u]_{\frac{a}{2m},a,Q_R^+}\le N(\|D^m u\|_{L_\infty(Q_{4R}^+)}+F),
 \end{equation*}
 where $N=N(d,m,n,\delta, \|A^{\alpha\beta}\|_a^\ast, K,R,a)$ and $F $ is defined in \eqref{except m} .
\end{corollary}
\begin{proof}[Proof of Theorem \ref{tan thm}:] Now we are ready to handle the general system
\begin{equation*}
u_t+(-1)^m \l u=\sum_{|\alpha|\le m}D^\alpha f_\alpha,
\end{equation*}
which can be written as
\begin{equation*}
u_t+(-1)^m\l_h u=\sum_{|\alpha|\le m}D^\alpha \tilde{f}_\alpha,
\end{equation*}
where
\begin{align*}
\tilde{f}_\alpha=f_\alpha+(-1)^{m+1}\sum_{|\beta|<m}A^{\alpha\beta}D^\beta u\quad \text{if}\quad |\alpha|=m,\\
\tilde{f}_\alpha=f_\alpha+(-1)^{m+1}\sum_{|\beta|\le m}A^{\alpha\beta}D^\beta u\quad \text{otherwise}.
\end{align*}
Note that
\begin{equation*}
\sum_{|\beta|<m}[A^{\alpha\beta}D^\beta u]^\ast_{a, Q_{4R}^+}\le N\sum_{|\beta|<m}(\|A^{\alpha\beta}\|_{L_\infty}[D^\beta u]^\ast_{a,Q_{4R}^+}+[A^{\alpha\beta}]^\ast_{a}\|D^\beta u\|_{L_\infty(Q_{4R}^+)}).
\end{equation*}
 We substitute $f_\alpha$ with $\tilde{f}_\alpha$ in the estimate of Corollary \ref{cor tan} to get
 $$[D_{x^\prime}D^{m-1}u]_{\frac{a}{2m},a, Q_R^+}\le N(\|D^m u\|_{L_\infty(Q_{4R}^+)}+\sum_{|\beta|<m}\|D^\beta u\|^\ast_{a, Q_{4R}^+}+F).$$
 After implementing a standard interpolation inequality, for instant see Section 8.8 of \cite{Kry96} and Lemma 5.1 of \cite{Gia93}, we prove Theorem \ref{tan thm}.
\end{proof}


 \subsection{The estimate of $D_d^m u$}
 In this subsection, we prove the Schauder estimate for $D_d^m u$. In this case, in view of the example in the introduction it is not sufficient that the coefficients are merely measurable in $t$.  We need the coefficients to be H\"older continuous in $t$ as well.

\subsubsection{Special system with coefficients depending only on $t$}
In this part, we study the special system which we introduced in Lemma \ref{boundary condition}.
\begin{lemma}\label{const coeff}
Assume that $u\in C_{loc}^\infty(\overline{\o_\infty^+})$ satisfies
\begin{equation*}
u_t+(-1)^m \tilde{\l}_0 u= \sum_{|\alpha|=m}D^\alpha f_\alpha \quad \text{in}\quad Q_{2R}^+
\end{equation*}
with the conormal derivative boundary condition on $\{x_d=0\}\cap Q_{2R}$. Moreover, $A^{\hat{\alpha}\hat{\alpha}}$ and $f_{\hat{\alpha}}$ are constants, and $f_\alpha=f_\alpha(t)$ for $|\alpha|=m, \alpha\neq\hat{\alpha}$. Then for any $0<r<R\le 1$ and $X_0\in \{x_d=0\}\cap Q_R$, there exist  a constant $N=N(d,n,m,\delta, K)$ and a constant vector $c$ such that
\begin{equation*}
\int_{Q_{r}^+(X_0)}|D^m_d u-c|^2\,dx\,dt \le N(\frac{r}{R})^{4m+d}\int_{Q_{R}^+(X_0)}|D_d^m u-c|^2\,dx\,dt.
\end{equation*}
\end{lemma}
\begin{proof}
Following the method in the proof of Lemma \ref{boundary condition}, it is obvious that the  system and boundary condition can be written as
\begin{align*}
&u_t+(-1)^m \tilde{\l}_0 u=0 \quad \text{in}\quad Q_{2R}^+,\\
&D_d^m u=(-1)^m({A^{\hat{\alpha}\hat{\alpha}}})^{-1}f_{\hat{\alpha}},\, D_d^{m+1}u=\cdots=D_d^{2m-1}u=0\,\, \text{on}\, \{x_d=0\}\cap Q_{2R}.
\end{align*}
Due to our assumption, $D_d^m u$ is constant on the boundary.  Hence $v:=D_d^m u-(-1)^m({A^{\hat{\alpha}\hat{\alpha}}})^{-1}f_{\hat{\alpha}}$ {satisfies}
\begin{align*}
&v_t+(-1)^m \tilde{\l}_0 v=0 \quad \text{in}\quad Q_{2R}^+,\\
&v=\cdots=D_d^{m-1}v=0\quad \text{on}\quad \{x_d=0\}\cap Q_{2R}.
\end{align*}
From Lemma 4.6 of \cite{DZ12},
\begin{equation*}
\int_{Q_{r}^+(X_0)}|v|^2\,dx\,dt \le N(\frac{r}{R})^{4m+d}\int_{Q_{R}^+(X_0)}|v|^2\,dx\,dt.
\end{equation*}
Setting $c=(-1)^m({A^{\hat{\alpha}\hat{\alpha}}})^{-1}f_{\hat{\alpha}}$, we obtain the desired estimate.
\end{proof}
We now use a scaling argument to handle $D_d^m u$ for systems with simple coefficients. Assume that $u$ satisfies \eqref{non homo simple} and as in the proof of Theorem \ref{theorem tan}, we define $v, \tilde{f}_\alpha$, and ${\tilde A}^{\alpha\beta}$.  Then $v$ satisfies the following system
\begin{equation}
v_t+(-1)^m\sum_{|\alpha|=|\beta|=m}\mu^{\alpha_d+\beta_d-2m}D^\alpha({\tilde A}^{\alpha\beta}(t)D^\beta v)=\sum_{|\alpha|\le m}\mu^{\alpha_d-2m}D^\alpha \tilde{f}_\alpha  \label{eq v}
\end{equation}
in $T_\mu(Q_{4R}^+)$, with the conormal derivative boundary condition on $\{x_d=0\}\cap T_\mu(Q_{4R})$, where $T_\mu(t,x^\prime,x_d)=(\mu^{2m}t, x^\prime, \mu x_d)$. From now on, we consider the system of $v$ with $\mu\ge 1$ and it is easy to see that $Q^+_{4R}\subset T_\mu(Q^+_{4R})$.  Note that the regularity assumptions on $A^{\alpha\beta}$ and $f_\alpha$ are naturally inherited by ${\tilde A}^{\alpha\beta}$ and $\tilde{f}_\alpha$. For instance,  if $f_{\hat{\alpha}}\in C^{\frac{a}{2m},a}$, then $\tilde{f}_{\hat{\alpha}}\in C^{\frac{a}{2m},a}$.

\begin{lemma}\label{nor holder1}
Assume that $a\in (0,1)$, $v\in C_{loc}^{\infty}(\overline{{\o_\infty^+}})$ satisfies \eqref{eq v} in $Q^+_{4R}$ with the conormal derivative boundary condition on $\{x_d=0\}\cap Q_{4R}$,  $\tilde{f}_{\hat{\alpha}} \in C^{\frac{a}{2m},a}$ and $\tilde{f}_\alpha\in C^{a\ast}$ if $|\alpha| =m$ but $\alpha\neq\hat{\alpha}$, $\tilde{f}_\alpha \in L_\infty$ if $|\alpha|< m$. Moreover, ${\tilde{A}}^{\hat{\alpha}\beta}(t) \in C^{\frac{a}{2m}}(\bR) $. Then there exist two constants $N_1=N_1(d,n,m,\delta,K,\langle{\tilde{A}}^{\hat{\alpha}\beta}\rangle_{\frac{a}{2m}},R,a)$ and $N_2=N_2(d,n,m,\delta, K,R,a)$ such that
\begin{align*}
[D_d^m v]_{\frac{a}{2m},a,Q_R^+}\le& N_1\|D^m v\|_{L_\infty(Q_{4R}^+)}+N_2([D_{x^\prime}D^{m-1} v]_{\frac{a}{2m},a,Q_{4R}^+(X_0)}\\
&+\mu^{-1}[D_d^m v]_{\frac{a}{2m},a,Q_{4R}^+}+\tilde{G}),
\end{align*}
where
\begin{equation}
{\tilde{G}=\sum_{^{|\alpha|=m}_{\alpha\neq\hat{\alpha}}}
[\tilde{f}_\alpha]^{\ast}_{ a, Q_{4R}^+}+{[\tilde{f}_{\hat{\alpha}}]_{\frac{a}{2m},a,Q_{4R}^+}}
+\sum_{|\alpha|\le m}\|\tilde{f}_\alpha\|_{L_\infty(Q_{4R}^+)}\label{def G}.}
\end{equation}
\end{lemma}
\begin{proof}
 We move all the spatial derivatives to the right-hand side and then add $$(-1)^m\cD_{d-1}^mv+(-1)^m D^{\hat{\alpha}}(\tilde{A}^{\hat{\alpha}\hat{\alpha}}(0)D^{\hat{\alpha}}v)$$ to both sides of the system so that
\begin{align}
&v_t+(-1)^m \big(D^{\hat{\alpha}}({\tilde A}^{\hat{\alpha}\hat{\alpha}}(0)D^{\hat{\alpha}}v)+\cD_{d-1}^mv\big)\nonumber\\
&=\sum_{|\alpha|\le m}D^\alpha \hat{f}_\alpha+(-1)^m\cD_{d-1}^mv+(-1)^{m+1}D^{\hat{\alpha}}(({\tilde A}^{\hat{\alpha}\hat{\alpha}}(t)-{\tilde A}^{\hat{\alpha}\hat{\alpha}}(0))D^{\hat{\alpha}}v)\label{const_1},
\end{align}
where $\hat{f}_\alpha=\mu^{\alpha_d-2m}\tilde{f}_\alpha$ for $|\alpha|<m$,
$$\hat{f}_\alpha=\mu^{\alpha_d-2m}\tilde{f}_\alpha+ (-1)^{m+1}\sum_{{|\beta|=m}}\mu^{\alpha_d+\beta_d-2m} {\tilde{A}}^{\alpha\beta}(t)D^\beta v$$ {for} $|\alpha|=m$ {but} $\alpha\neq \hat{\alpha}$, and
$$\hat{f}_{\hat{\alpha}}=\mu^{-m}\tilde{f}_{\hat{\alpha}}+(-1)^{m+1}\sum_{^{\beta\neq\hat{\alpha}}_{|\beta|=m}}\mu^{\beta_d-m}\tilde{A}^{\hat{\alpha}\beta}(t)D^\beta v.$$
Notice that the left-hand side of \eqref{const_1} satisfies the conditions in Lemma \ref{const coeff} and similar to $\tilde{\l}_0$ we denote this operator as $\hat{\l}_0$.  The only difference is that $A^{\hat{\alpha}\hat{\alpha}}$ is constant in $\hat{\l}_0$.

We use basically the same argument as in Lemma \ref{tan holder} with slight modifications.
Let $w\in \h_2^m((-(4R)^{2m},0)\times \bR^{d}_+)$ be the solution of the following system
\begin{align}\nonumber
&w_t+(-1)^m \hat{\l}_0 w=\sum_{|\alpha|\le m}D^{\alpha}(\zeta\overline{f}_\alpha)+(-1)^m \sum_{j=1}^{d-1}D_j^m(\zeta(D^m_j v-D^m_jv(t,0)))\\
&+(-1)^{m+1}D^{\hat{\alpha}}(\zeta(\tilde{A}^{\hat{\alpha}\hat{\alpha}}(t)-\tilde{A}^{\hat{\alpha}\hat{\alpha}}(0))D^{\hat{\alpha}}v)\quad\text{in} \quad(-(4R)^{2m},0) \times \bR_+^{d}\label{eq 3.38}
\end{align}
with the conormal derivative boundary condition on $\{x_d=0\}$ and the zero initial condition on $\{-(4R)^{2m}\}\times\bR^{d}_+$, where $\zeta$ is the same smooth function as in Lemma \ref{tan holder}, $\overline{f}_{\hat{\alpha}}=\hat{f}_{\hat{\alpha}}(t,x)-\hat{f}_{\hat{\alpha}}(0,0)$, $\overline{f}_\alpha=\hat{f}_\alpha(t,x)-\hat{f}_\alpha(t,0)$ if $|\alpha|=m$ and $\alpha\neq\hat{\alpha}$, and $\overline{f}_\alpha=\hat{f}_\alpha$ otherwise. Upon applying the mollification argument as in Lemma \ref{mean half space}, we can assume that $w$ is smooth. Then $\mathfrak{v}:=u-w$, satisfies
\begin{equation}
\mathfrak{v}_t+(-1)^m\hat{\l}_0 \mathfrak{v}=\sum_{^{|\alpha|=m}_{\alpha\neq\hat{\alpha}}}D^{\alpha}\hat{f}_\alpha(t,0)+D^{\hat{\alpha}}\hat{f}_{\hat{\alpha}}(0,0)+(-1)^m\sum_{j=1}^{d-1}D_j^{2m}v(t,0) \label{eq 3.39}
\end{equation}
in $Q_{2R}^+$ with the conormal derivative boundary condition on $\{x_d=0\}\cap Q_{2R}$.
 Due to Lemma \ref{const coeff},
\begin{equation*}
\int_{Q_{r}^+(X_0)}|D^m_d \mathfrak{v}-c|^2\,dx\,dt \le N(\frac{r}{R})^{4m+d}\int_{Q_{R}^+(X_0)}|D_d^m \mathfrak{v}-c|^2\,dx\,dt
\end{equation*}
for some constant vector $c$.
Next we estimate $w$  applying the same idea as in Lemma \ref{tan holder},
\begin{align}\nonumber
&\delta\int_{Q_{4R}^+} |D^m w|^2 \,dx\,dt \\\nonumber
&\le\sum_{|\alpha|<m}(-1)^{|\alpha|}\int_{Q_{4R}^+}\zeta(x) \mu^{\alpha_d-2m}\tilde{f}_\alpha D^\alpha w\,dx\,dt\\ \nonumber
&\quad -\int_{Q_{4R}^+}\zeta(\tilde{A}^{\hat{\alpha}\hat{\alpha}}(t)-\tilde{A}^{\hat{\alpha}\hat{\alpha}}(0))D^{\hat{\alpha}}vD^{\hat{\alpha}}w\,dx\,dt
\\\nonumber
 &\quad +(-1)^m\int_{Q_{4R}^+}\zeta(x)(\hat{f}_{\hat{\alpha}}(t,x)-\hat{f}_{\hat{\alpha}}(0,0))D^{\hat{\alpha}} w \,dx\,dt\\\nonumber
&\quad +N\sum_{^{|\alpha|=m}_{\alpha\neq\hat{\alpha}}}(-1)^m\int_{Q_{4R}^+}\zeta(x)(\hat{f}_\alpha(t,x)-\hat{f}_\alpha(t,0))D^\alpha w \,dx\,dt\\
&\quad +\sum_{j=1}^{d-1}\int_{Q_{4R}^+}\zeta (D_j^mv(t,x)-D_j^mv(t,0))D^m_j w \,dx\,dt.\label{eq 7.41}
\end{align}
The first term on the right-hand side of \eqref{eq 7.41} can be dealt with in the same way as \eqref{eq 7.37} in Lemma \ref{tan holder}. Indeed, we use the same $h$ as in the proof of Lemma \ref{tan holder}, i.e., $h=w-P(x)$ where $P(x)$ is the vector-valued polynomial as in Lemma \ref{tech 3}. Therefore, the inequality above holds with $h$ in place of $w$ in the first term on the right-hand side of the inequality above and the first term can be estimated as follows:
\begin{align*}
&\int_{Q_{4R}^+}|\zeta \mu^{\alpha_d-2m}\tilde{f}_\alpha D^\alpha h |\,dx\,dt\le N\mu^{\alpha_d-2m}\|\tilde{f}_\alpha\|_{L_\infty(Q_{4R}^+)}|Q_{4R}^+|^\.5\|D^\alpha h\|_{L_2(Q_{4R}^+)}\nonumber.
\end{align*}
Then we apply Lemma \ref{tech 3} to $h$, {which satisfies}  \eqref{eq 3.38}, and note that $\tilde{A}^{\hat{\alpha}\hat{\alpha}}\in C^{\frac{a}{2m}}(\bR)$ to {bound}  the first term {by}
\begin{align*}
& N \mu^{\alpha_d-2m}\|\tilde{f}_\alpha\|_{L_\infty(Q_{4R}^+)}|Q_{4R}^+|^\.5 \big(R^{m-|\alpha|}\|D^m w\|_{L_2(Q_{4R}^+)}\\
&+\sum_{|\beta|\le m}R^{3m+d/2-|\alpha|-|\beta|}\|\overline{f}_\beta\|_{L_\infty(Q_{4R}^+)}\\
&+R^{2m+d/2-|\alpha|+a}([D_{x^\prime}^m v]_{a,Q_{4R}^+}^\ast+\|D_d^m v\|_{L_\infty(Q_{4R}^+)})\big).
\end{align*}
Recall the definition of $\overline{f}_\alpha$ and notice that
$$\|\overline{f}_\alpha\|_{L_\infty(Q_{4R}^+)}\le [\tilde{f}_\alpha]_{a,Q_{4R}^+}^\ast+N_2[D_{x^\prime}D^{m-1}v]^*_{a,Q_{4R}^+}+N_2\mu^{-1}[D_d^m v]^*_{a,Q_{4R}^+}$$
for $|\alpha|=m$ {but} $\alpha\neq \hat{\alpha}$, and
\begin{align*}
&\|\overline{f}_{\hat{\alpha}}\|_{L_\infty(Q_{4R}^+)}\\
&\le [\tilde{f}_{\hat{\alpha}}]_{\frac{a}{2m},a,Q_{4R}^+}+N_2\langle{\tilde{A}}^{\hat{\alpha}\beta}\rangle_{\frac{a}{2m}}\|D^m v\|_{L_\infty(Q_{4R}^+)}+N_2[D_{x^\prime}D^{m-1} v]_{\frac{a}{2m},a,Q_{4R}^+}.
\end{align*}
We obtain
\begin{align*}
&\sum_{|\alpha|<m}\int_{Q_{4R}^+}|\zeta(x)\mu^{\alpha_d-2m}\tilde{f}_\alpha D^\alpha w|\,dx\,dt\\
&\le \varepsilon\|D^m w\|^2_{L_2(Q_{4R}^+)}+N_2(\varepsilon)\sum_{|\alpha|<m}\|\tilde{f}_\alpha\|^2_{L_\infty(Q_{4R}^+)}R^{2+2m+d}\\
&\quad+N_2(\varepsilon)(\sum_{^{|\alpha|=m}_{\alpha\neq\hat{\alpha}}}[\tilde{f}_\alpha]_{a, Q_{4R}^+}^{\ast2}+[\tilde{f}_{\hat{\alpha}}]^2_{\frac{a}{2m},a,Q_{4R}^+}
+[D_{x^\prime}D^{m-1}v]^{2}_{\frac{a}{2m}, a, Q_{4R}^+}\\
&\quad+\mu^{-1}[D_d^m v]^2_{\frac{a}{2m},a,Q_{4R}^+})R^{2a+d+2m}+N_1\|D^m v\|^2_{L_\infty(Q_{4R}^+)}R^{2a+d+2m}
\end{align*}
provided that $\mu{\ge 1}$.

For the other terms on the right-hand side of \eqref{eq 7.41}, following the proof of Lemma \ref{tan holder}, we apply Young's inequality so that, for any $\varepsilon>0$, the right-hand side of \eqref{eq 7.41} is bounded by
\begin{align}\nonumber
& \varepsilon \|D^m w\|^2_{L_2(Q_{4R}^+)}+N_2(\varepsilon)\big(\sum_{^{|\alpha|=m}_{\alpha\neq\hat{\alpha}}}[\tilde{f}_\alpha]_{a, Q_{4R}^+}^{\ast2}+[\tilde{f}_{\hat{\alpha}}]^2_{\frac{a}{2m},a,Q_{4R}^+}\\
&+[D_{x^\prime}D^{m-1} v]^{2}_{\frac{a}{2m},a,Q_{4R}^+}+\mu^{-1}[D_d^m v]^2_{\frac{a}{2m},a, Q_{4R}^+}\big)R^{2a+2m+d}\nonumber\\
&+N_1\|D^m v\|^2_{L_\infty(Q_{4R}^+)}R^{2a+2m+d}
+N_2\sum_{|\alpha|<m}\|\tilde{f}_\alpha\|^2_{L_\infty(Q_{4R}^+)}R^{2+2m+d}.\label{eq 7.42}
\end{align}
After choosing $\varepsilon$ sufficiently small, by the triangle inequality, \eqref{eq 3.39} and \eqref{eq 7.42} we obtain
\begin{align*}
&\int_{Q_r^+(X_0)}|D^m_d v-c|^2\,dx\,dt\\
&\le N_2 (\frac{r}{R})^{4m+d}\int_{Q_R^+(X_0)}|D^m_d v-c|^2\,dx\,dt+IR^{2a+2m+d},
\end{align*}
where
\begin{align*}
I:=&N_2(\sum_{^{|\alpha|=m}_{\alpha\neq\hat{\alpha}}}[\tilde{f}_\alpha]^{\ast2}_{ a, Q_{4R}^+} +N_1\|D^m v\|^2_{L_\infty(Q_{4R}^+)}+\sum_{|\alpha|<m} \|\tilde{f}_\alpha\|_{L_\infty(Q_{4R}^+)}^2 \\
&+[\tilde{f}_{\hat{\alpha}}]^2_{\frac{a}{2m},a,Q_{4R}^+}+[D_{x^\prime}D^{m-1} v]^{2}_{\frac{a}{2m},a,Q_{4R}^+(X_0)}+\mu^{-1}[D_d^m v]^2_{\frac{a}{2m},a,Q_{4R}^+}).\\
 \end{align*}
From Lemma \ref{tech}, we know that
\begin{align*}
&\int_{Q_r^+(X_0)}|D^m_d v-c|^2\,dx\,dt\\
&\le N_2 (\frac{r}{R})^{2a+2m+d}\int_{Q_R^+(X_0)}|D^m_d v-c|^2\,dx\,dt+Ir^{2a+2m+d}.
\end{align*}
By Lemma \ref{holder}, we get
\begin{align*}
&[D_d^m v]^2_{\frac{a}{2m},a,Q_R^+}\le N_2(R^{-(2a+d+2m)}\|D^m v-c\|^2_{L_2(Q_{4R}^+)}+I)\\
&\le N_2([D_{x^\prime}D^{m-1}v]^{2}_{\frac{a}{2m},a,Q_{4R}^+}+\mu^{-1}[D_d^m v]^{2}_{\frac{a}{2m},a,Q_{4R}^+}+\tilde{G}^{2})+N_1\|D^m v\|^{2}_{L_\infty(Q_{4R}^+)}.
\end{align*}
Therefore, we prove the lemma.
\end{proof}

\subsubsection{General systems with coefficients depending on both $x$ and $t$.}
Similar to Lemma \ref{nor holder1}, we can estimate the highest normal derivative in the case of variable coefficients depending on both $x$ and $t$. As before, we need more regularity assumptions on $\tilde{A}^{\alpha\beta}$ and $\tilde{f}_\alpha$.  Similar to Lemma \ref{high tan} and Corollary \ref{cor tan}, following the proof of Lemma \ref{nor holder1}, we can prove the lemma below.

\begin{lemma}\label{high nor}
Assume that $a\in(0,1)$, $v\in C_{loc}^{\infty}(\overline{{\o_\infty^+}})$ satisfies
\begin{equation*}
v_t+(-1)^m \sum_{|\alpha|=|\beta|=m}\mu^{\alpha_d+\beta_d-2m}D^\alpha(\tilde{A}^{\alpha\beta}D^\beta v)=\sum_{|\alpha|\le m}\mu^{\alpha_d-2m}D^\alpha \tilde{f}_\alpha
\end{equation*}
in $Q_{4R}^+$ with the conormal derivative boundary condition on $\{x_d=0\}\cap Q_{4R}$ and  $\tilde{f}_{\hat{\alpha}} \in C^{\frac{a}{2m}, a}, \tilde{f}_\alpha\in C^{a\ast}$ if $|\alpha| =m$ and $\alpha\neq\hat{\alpha}$, $\tilde{f}_\alpha \in L_\infty$ if $|\alpha|< m$, $\tilde{A}^{\hat{\alpha}\beta}\in C^{\frac{a}{2m}, a}$, and $\tilde{A}^{\alpha\beta}\in C^{a\ast}$ for the other $|\alpha|=m$. Then there exist two constants $N_1=N_1(d,n,m,\delta,K,\|\tilde{A}^{\hat{\alpha}\beta}\|_{\frac{a}{2m}, a}, \|\tilde{A}^{\alpha\beta}\|_{a}^\ast, R,a)$ and $N_2=N_2(d,n,m,\delta,K,R,a)$ such that
\begin{align*}
[D_d^m v]_{\frac{a}{2m},a,Q_{R}^+}&\le N_1\|D^m v\|_{L_\infty(Q_{4R}^+)}+N_2([D_{x^\prime}D^{m-1}v]_{\frac{a}{2m},a,Q_{4R}^+}\\
&\quad+\mu^{-1}[D_d^m v]_{{\frac{a}{2m},a,}Q_{4R}^+}+ \tilde{G}),
\end{align*}
where $\tilde{G}$ is defined in \eqref{def G}.
\end{lemma}

Further{more}, from Lemma \ref{high nor} we follow the proof of the Theorem \ref{tan thm} moving the lower-order terms to the right-hand side regarded as part of $\tilde{f}_\alpha$ and noticing that
\begin{align*}
&\sum_{|\beta|<m}[\tilde{A}^{\alpha\beta}D^\beta v]_{\frac{a}{2m},a, Q_{4R}^+}\\&\le N\sum_{|\beta|<m}(\|\tilde{A}^{\alpha\beta}\|_{L_\infty}[D^\beta v]_{\frac{a}{2m},a,Q_{4R}^+}+[\tilde{A}^{\alpha\beta}]_{\frac{a}{2m},a}^{}\|D^\beta v\|_{L_\infty(Q_{4R}^+)})
\end{align*}
to get the following lemma.
\begin{lemma}\label{nor v}
Assume that $a\in(0,1)$, $v\in C_{loc}^{\infty}(\overline{{\o_\infty^+}})$ satisfies
\begin{equation*}
v_t+(-1)^m \sum_{|\alpha|\le m,|\beta|\le m}\mu^{\alpha_d+\beta_d-2m}D^\alpha(\tilde{A}^{\alpha\beta}D^\beta v)=\sum_{|\alpha|\le m}\mu^{\alpha_d-2m}D^\alpha \tilde{f}_\alpha
\end{equation*}
in $Q_{4R}^+$ with the conormal derivative boundary condition on $\{x_d=0\}\cap Q_{4R}$ and  $\tilde{f}_{\hat{\alpha}} \in C^{\frac{a}{2m}, a}, \tilde{f}_\alpha\in C^{a\ast}$ if $|\alpha| =m$ and $\alpha\neq\hat{\alpha}$, $\tilde{f}_\alpha \in L_\infty$ if $|\alpha|< m$, $\tilde{A}^{\hat{\alpha}\beta}\in C^{\frac{a}{2m}, a}$, and $\tilde{A}^{\alpha\beta}\in C^{a\ast}$ for the other $|\alpha|=m$. Then there exist two constants $N_1=N_1(d,n,m,\delta,K,\|\tilde{A}^{\hat{\alpha}\beta}\|_{\frac{a}{2m}, a}, \|\tilde{A}^{\alpha\beta}\|_{a}^\ast, R,a)$ and $N_2=N_2(d,n,m,\delta,K,R,a)$ such that
\begin{align*}
[D_d^m v]_{\frac{a}{2m},a,Q_{R}^+}&\le N_1\sum_{|\beta|\le m}\|D^\beta v\|_{L_\infty(Q_{4R}^+)}+N_2([D_{x^\prime}D^{m-1}v]_{\frac{a}{2m},a,Q_{4R}^+}
\\&\quad+\mu^{-1}[D_d^m v]_{{\frac{a}{2m},a,}Q_{4R}^+}+\sum_{|\beta|<m}[D^\beta v]_{\frac{a}{2m},a,Q_{4R}^+}+\tilde{G}),
\end{align*}
where $\tilde{G}$ is defined in \eqref{def G}.
\end{lemma}

\begin{remark}\label{tan v}Note that we can estimate $D_{x^\prime}D^{m-1}v$ as well. In fact, all the proofs hold with $v$ in place of $u$. The only difference is that the constant $N$ in Theorem \ref{tan thm} may depend on $\mu$.
\end{remark}

\section{Proof of Theorem \ref{thm schauder}}
The {following} interior Schauder estimates of divergence type higher-order parabolic systems have been established in \cite{DZ12}.
\begin{proposition}\label{prop 8.1}
Assume that $R\in (0, 1]$ and $u\in C_{loc}^\infty(\bR^{d+1})$ satisfies
\begin{equation*}
u_t+(-1)^m \l u=\sum_{|\alpha|\le m}D^\alpha f_\alpha \quad \text{in}\quad Q_{2R}.
\end{equation*}
Suppose that $A^{\alpha\beta}\in C^{a\ast}$, $f_\alpha\in C^{a\ast}$ if $|\alpha|=m$ where $a\in(0,1)$, and $f_\alpha \in L_\infty$ if $|\alpha|<m$. Then there exists a constant $N(d,n,m,\lambda,K, \|A^{\alpha\beta}\|_a^\ast, R, {a})$ such that
\begin{align*}
\langle u\rangle_{\frac{1}{2}+\frac{a}{2m}, Q_R} +[D^m u]_{\frac{a}{2m},a,Q_R}&\le N(\sum_{|\beta|\le m}\|D^{\beta}u\|_{L_\infty(Q_{2R})}+F),
\end{align*}
{where $$F=\sum_{|\alpha|=m}[f_\alpha]_{a, {Q_{2R}}}^\ast+\sum_{|\alpha|<m}\|f_\alpha\|_{L_\infty({Q_{2R}})}.$$}
\end{proposition}
We have the following global estimate in $(0,T)\times \bR^{d}_+$.
\begin{proposition}\label{final pro}
Let $a\in (0,1)$, $u\in C_{loc}^{\infty}(\overline{\bR_+^{d+1}})$ satisfy
\begin{equation*}
u_t+(-1)^m\l u=\sum_{|\alpha|\le m}D^\alpha f_\alpha \quad \text{in}\quad (0,T)\times\bR^{d}_+
\end{equation*}
with the conormal derivative boundary condition on $(0,T)\times \{x:x_d=0\}$, and $f_{\hat{\alpha}} \in C^{\frac{a}{2m}, a}, f_\alpha\in C^{a\ast}$ if $|\alpha| =m$ and $\alpha\neq\hat{\alpha}$, $f_\alpha \in L_\infty$ if $|\alpha|< m$, $A^{\hat{\alpha}\beta}\in C^{\frac{a}{2m}, a}$, and $A^{\alpha\beta}\in C^{a\ast}$ for the other $\alpha$. Then there exists a constant $N=N(d,n,m,\delta,K,\|A^{\hat{\alpha}\beta}\|_{\frac{a}{2m},a},\|A^{\alpha\beta}\|_a^\ast, a)$ such that
\begin{align*}
&[D^m u]_{\frac{a}{2m},a,(0,T)\times \bR^{d}_+}\\
&\le N(\|D^m u\|_{L_\infty((0,T)\times\bR^{d}_+)}+\sum_{|\beta|<m}\|D^\beta u\|_{\frac{a}{2m},a, (0,T)\times\bR^{d}_+}+G).
\end{align*}
\end{proposition}
\begin{proof}
We use a scaling argument and consider the equation of $v$ in the corresponding domain $T_\mu((0,T)\times \bR^{d}_+)$. If we can prove the inequality above for $v$, then after changing back to $u$, we prove the lemma. Applying the argument of partition of the unity, translation of the coordinates,  Lemma \ref{nor v}, Remark \ref{tan v} and the interior estimate Proposition \ref{prop 8.1},  we know that
\begin{align*}
&[D^m v]_{\frac{a}{2m},a, (0,{T_1})\times \bR^{d}_+}\\
&\le N_1(\mu)(\|D^m v\|_{L_\infty((0,T_1)\times \bR^{d}_+)}+\sum_{|\beta|<m}\|D^\beta v\|_{\frac{a}{2m},a, (0,T_1)\times \bR^{d}_+})\\
&\quad+N_2(\mu^{-1}[D_d^m v]_{\frac{a}{2m},a, (0,T_1)\times \bR^{d}_+}+\tilde{G}),
\end{align*}
where $N_1(\mu)$ depends on $\mu$, but $N_2$ does not, and $T_1=\mu^{2m}T$. Let $\mu$ be sufficiently large such that $N_2\mu^{-1}[D_d^m v]_{\frac{a}{2m},a, (0,T)\times \bR^{d}_+}$is absorbed to the left-hand side. Then we fix this $\mu$ and obtain that
\begin{align*}
&[D^m v]_{\frac{a}{2m},a, (0,T_1)\times \bR^{d}_+}\\
&\le N(\|D^m v\|_{L_\infty((0,T_1)\times \bR^{d}_+)}+\sum_{|\beta|<m}\|D^\beta v\|_{\frac{a}{2m},a, (0,T_1)\times \bR^{d}_+}+\tilde{G}).
\end{align*}
Therefore, we prove the {proposition}.
\end{proof}

Moreover, we can estimate the regularity of $u$ in the $t$ variable.
First we state a lemma which is a particular case of Lemma 3.3 in \cite{DZ12}.
\begin{lemma}\label{tech function}
There exists a function $\xi\in C^{\infty}_0(B_1^+)$ with unit integral such that for any $0<|\alpha|\le m$,
\begin{equation*}
\int_{B_1^+}\xi(y)y^\alpha\, dy =0.
\end{equation*}
\end{lemma}

Next, we show the estimate in the $t$ variable by applying the method in Proposition 4.1 in \cite{DZ12} which corresponds to the interior estimates.
\begin{proposition}\label{estimate in t}
Suppose that $u \in C^\infty_{loc}(\overline{\mathbb{R}^{d+1}_+})$ and satisfies
\begin{equation*}
u_t+(-1)^m\l u=\sum_{|\alpha|\le m}D^\alpha f_\alpha \quad \text{in} \quad Q_{2R}^+
\end{equation*}
with the conormal derivative boundary condition on $\{x_d=0\}$.
The coefficients of $\l$ and $f_\alpha$ satisfy the same condition as in Theorem \ref{thm schauder}. Then there exists a constant $N=N(d,m,n,\delta,K,\|A^{\alpha\beta}\|^\ast_{a},R,a)$ such that
\begin{equation*}
\langle u\rangle_{\frac{1}{2}+\frac{a}{2m}, Q_R^+} \le N(\|D^m u\|^\ast_{a,Q_{2R}^+}+\sum_{|\beta|<m}\|D^\beta u\|_{L_\infty( Q_{2R}^+)}+F).
\end{equation*}
\end{proposition}

\begin{proof}
We estimate $|u(t,x_0)-u(s,x_0)|$, where $(t,x_0)$, $(s,x_0) \in Q_R^+$.
Let $\rho=|t-s|^{\frac{1}{2m}}$ and $\eta(x)=\frac{1}{\rho^d}\xi(\frac{x}{\rho})$, where $\xi(x)$ is the function in Lemma \ref{tech function}. We define
\begin{align*}
\overline{u}(t,y)= u(t,y)-T_{m,{x_0}}u(t,y)+u(t,x_0),
\end{align*}
where $T_{m, {x_0}}u(t,y)$ is the Taylor expansion of $u$ in $y$ at $x_0$ up to $m$th order. Therefore,
\begin{align}\nonumber
&|u(t,x_0)-u(s,x_0)|\le|u(t,x_0)-\int_{B^+_\rho(x_0)}\eta(x_0-y)\overline{u}(t,y)\,dy|\\\nonumber
&+|u(s,x_0)-\int_{B^+_\rho(x_0)}\eta(x_0-y)\overline{u}(s,y)\,dy|\\
&+|\int_{B^+_\rho(x_0)}\eta(x_0-y)\overline{u}(s,y)\,dy-\int_{B^+_\rho(x_0)}\eta(x_0-y)\overline{u}(t,y)\,dy|.\label{eq 7.15}
\end{align}
The first two terms on the right-hand side can be estimated in a similar fashion: noting that $\int_{B_\rho^+} \eta(x)\,dx=1$,
\begin{align*}
&|u(t,x_0)-\int_{B^+_\rho(x_0)}\eta(x_0-y)\overline{u}(t,y)\,dy|\\
&=|\int_{B^+_\rho(x_0)}\eta(x_0-y)(u(t,x_0)-\overline{u}(t,y))\,dy|.
\end{align*}
Since
\begin{align*}
|\overline{u}(t,y)-u(t,x_0)|=|u(t,y)-T_{m,x_0}u(t,y)|\le C [D^m u]^\ast_{a, Q_{2R}^+}|y-x_0|^{m+a},
\end{align*}
we get
\begin{align*}
|u(t,x_0)-\int_{B^+_\rho(x_0)}\eta(x_0-y)\overline{u}(t,y)\,dy|\le C[D^m u]^\ast_{a, Q_{2R}^+} \rho^{m+a}.
\end{align*}
For the last term of the right-hand side of \eqref{eq 7.15},  by the definition of $\eta$,
\begin{align*}
&|\int_{B^+_\rho(x_0)}\eta(x_0-y)\overline{u}(t,y)\,dy-\int_{B^+_\rho(x_0)}\eta(x_0-y)\overline{u}(s,y)\,dy|\\&=|\int_{B^+_\rho(x_0)}\eta(x_0-y)u(t,y)\,dy-\int_{B^+_\rho(x_0)}\eta(x_0-y)u(s,y)\,dy|\\
&=|\int_s^t \int_{B^+_\rho(x_0)}\eta(x_0-y)u_t(\tau,y)\,dy\,d\tau|.
\end{align*}
Because $u$ satisfies the equation,
\begin{align}\nonumber
&|\int_s^t \int_{B^+_\rho(x_0)}\eta(x_0-y)u_t(\tau,y)\,dy\,d\tau|\\&=|\int_s^t \int_{B^+_\rho(x_0)}\eta(x_0-y) (\sum_{|\alpha|\le m, |\beta|\le m}D^\alpha ((-1)^{m+1}A^{\alpha\beta}D^\beta u(\tau,y))\nonumber\\
&\quad+\sum_{|\alpha|\le m}D^\alpha f_\alpha)\,dy\,d\tau|\label{ut estimate}.
\end{align}
We can substitute $f_\alpha$ by $f_\alpha(t,x) -f_\alpha(t,x_0)$ when $|\alpha|=m$ since  $\eta$ has compact support in $B^+_\rho(x_0)$.  Moreover, for $|\alpha|=m$,
\begin{align*}
&D^\alpha (A^{\alpha\beta}D^\beta u)\\
&=D^\alpha \big((A^{\alpha\beta}-A^{\alpha\beta}(t,x_0))D^\beta u\big)+D^\alpha (A^{\alpha\beta}(t,x_0)(D^\beta u-D^\beta u(t,x_0))).
\end{align*}
We plug all these into \eqref{ut estimate} and integrate by parts. It follows easily that
\begin{align*}
&|\int_s^t \int_{B^+_\rho(x_0)}\eta(x_0-y)u_t(\tau,y)\,dy\,d\tau|\\
& \le \sum_{^{|\alpha|=m}_{|\beta|\le m}}\big([A^{\alpha\beta}]^\ast_{a}\|D^\beta u\|_{L_\infty(Q_{2R}^+)}+[D^\beta u]^\ast_{a, Q_{2R}^+}\|A^{\alpha\beta}\|_{L_\infty}\big)\\
&\quad\cdot \rho^a\int_{Q^+_\rho(X_0)}|D^m\eta(x_0-y)|\,dy\,d\tau \\
&\quad +\sum_{^{|\alpha|<m}_{|\beta|\le m}}\|A^{\alpha\beta}\|_{L_\infty}\|D^\beta u\|_{L_\infty(Q_{2R}^+)}\int_{Q_\rho^+(X_0)}|D^\alpha \eta(x_0-y)|\,dy\,d\tau\\
&\quad +\sum_{|\alpha|<m}\|f_\alpha\|_{L_\infty(Q_{2R}^+)}\int_{Q^+_\rho(X_0)}|D^\alpha \eta(x_0-y)|\,dy\,d\tau\\
&\quad +\sum_{|\alpha|=m}\|D^m \eta\|_{L_\infty}[f_\alpha]^\ast_{a, Q_{2R}^+}|Q^+_\rho|\rho^a,
\end{align*}
which is bounded by
\begin{align*}
&C\sum_{|\beta|=m}\|D^\beta u\|^\ast_{a, Q_{2R}^+}\rho^{m+a}+C\sum_{|\alpha|=m}[f_\alpha]^\ast_{a, Q_{2R}^+}\rho^{m+a}\\
&\quad+C\sum_{|\alpha|<m}(\|f_\alpha\|_{L_\infty(Q_{2R}^+)}+\|D^\alpha u\|_{L_\infty(Q_{2R}^+)})\rho^{m+1}.
\end{align*}
Here we used $\|D^\alpha \eta\|_{L_\infty}\le C \rho^{-d-|\alpha|}$ for any $|\alpha|\le m$.
Hence,
\begin{align*}
|u(t,x_0)-u(s,x_0)|\le C\rho^{m+a}(\|D^m u\|^\ast_{a,Q_{2R}^+}+\sum_{|\beta|< m}\|D^\beta u\|_{L_\infty(Q_{2R}^+)}+F).
\end{align*}
 The proof is completed.
\end{proof}

Following the idea in proving Theorem 2.1 of \cite{DZ12}, due to Proposition \ref{final pro}, Proposition \ref{estimate in t}, the standard arguments of partition of the unity and flattening the boundary, we obtain
\begin{align*}
&\langle u\rangle_{\frac{a+m}{2m},(0,T)\times \Omega}+[D^m u]_{\frac{a}{2m},a,(0,T)\times \Omega}\\
&\le N(\|D^m u\|_{L_\infty((0,T)\times \Omega)}+\sum_{|\beta|<m}\|D^\beta u\|_{\frac{a}{2m},a,(0,T)\times \Omega}+G).
\end{align*}
By the interpolation inequalities in H\"older spaces, for instance see Section 8.8 of \cite{Kry96},
\begin{equation*}
\langle u\rangle_{\frac{a+m}{2m},(0,T)\times \Omega}+[D^m u]_{\frac{a}{2m},a,(0,T)\times \Omega}\le N (\|u\|_{L_\infty((0,T)\times \Omega)}+G).
\end{equation*}
Applying the interpolation inequalities again, we get
\begin{equation*}
\|u\|_{L_\infty((0,T)\times \Omega)}\le N(\varepsilon)\|u\|_{L_2((0,T)\times \Omega)}+\varepsilon([D^m u]_{\frac{a}{2m},a,(0,T)\times \Omega}+\langle u\rangle_{\frac{a+m}{2m},(0,T)\times \Omega}).
\end{equation*}
Upon taking $\varepsilon$ sufficiently small, we arrive at
\begin{equation*}
\|u\|_{\frac{a+m}{2m}, a+m,(0,T)\times \Omega}\le N(\|u\|_{L_2((0,T)\times \Omega)}+G),
\end{equation*}
which is \eqref{finaldiv}.
In order to implement the continuity argument to prove the existence of solutions, we need the right-hand side of \eqref{finaldiv} to be independent of $u$ and this leads us to consider the following system:
\begin{align}\label{exist}
&u_t+(-1)^m\l u+\kappa u=\sum_{|\alpha|\le m}D^\alpha f_\alpha\quad \text{in}\quad (0,T)\times \Omega
\end{align}
with the conormal derivative boundary condition on $(0,T)\times\partial \Omega$ and the zero initial condition on $\{0\}\times \Omega$.
We choose $\kappa$ large enough, such that
\begin{equation}\label{l2control}
\|u\|_{L_2((0,T)\times \Omega)} \le NG,
\end{equation}
where $N=N(d,n,m,\lambda,K,\|A^{\alpha\beta}\|_a^\ast,\Omega)$.
To prove \eqref{l2control}, we rewrite the system as
\begin{equation*}
u_t+(-1)^m\l_h u+\kappa u=\sum_{|\alpha|\le m}D^\alpha f_\alpha +(-1)^{m+1}(\l-\l_h)u.
\end{equation*}
By the definition of the conormal derivative boundary condition and the ellipticity of $A^{\alpha\beta}$,
\begin{align}\nonumber
&\delta\|D^m u\|^2_{L_2((0,T)\times \Omega)} +\kappa \|u\|^2_{L_2((0,T)\times \Omega)} \\\nonumber
&\le \sum_{|\alpha|\le m}\int_{(0,T)\times \Omega} (-1)^{|\alpha|} f_\alpha D^\alpha u\,dx\,dt\\
&\quad-\sum_{|\alpha|=m,|\beta|<m} \int_{(0,T)\times \Omega}A^{\alpha\beta}D^\beta u D^\alpha u \,dx\,dt\nonumber\\
&\quad+(-1)^{m+1+|\alpha|}\sum_{|\alpha|<m,|\beta|\le m}\int_{(0,T)\times \Omega}A^{\alpha\beta}D^\beta u D^\alpha u \,dx\,dt.\label{eq 8.48}
\end{align}
We use the Schwarz inequality, Young's inequality, and the interpolation inequality to bound the right-hand side by
\begin{align*}
 &N(n,m,d,\delta,K,\Omega, \varepsilon) (\sum_{|\alpha|\le m}\|f_\alpha\|^2_{L_2((0,T)\times \Omega)}+\|u\|^2_{L_2((0,T)\times \Omega)})\\&+\varepsilon\|D^m u\|^2_{L_2((0,T)\times \Omega)}.
\end{align*}
After taking $\varepsilon$ sufficiently small to absorb the term $\varepsilon\|D^m u\|^2_{L_2((0,T)\times \Omega)}$ to the left-hand side of \eqref{eq 8.48} and choosing $\kappa$ sufficiently large, we reach \eqref{l2control}.
Combining \eqref{l2control} and \eqref{finaldiv}, we get the following lemma.
\begin{lemma}\label{lm8.4}
Assume that $\l, \Omega$, and $f_\alpha$ satisfy all the conditions in Theorem \ref{thm schauder} and $u \in C_{loc}^\infty(\mathbb{R}^{d+1}) $ satisfies \eqref{exist}, with the conormal derivative boundary condition {on $(0,T)\times \partial\Omega$}, the zero initial condition on $\{0\}\times \Omega$ and $\kappa$ large enough,
then
\begin{equation}\label{finexist}
\|u\|_{\frac{{a+m}}{2m},{a+m},{(0,T)\times \Omega}} \le NG,
\end{equation}
where $N=N(d, m, n, \delta,K,\|A^{\alpha\beta}\|_{\frac{a}{2m},a},\Omega)$.
\end{lemma}

Now we are ready to prove Theorem \ref{thm schauder}.
\begin{proof}[Proof of Theorem \ref{thm schauder}]
  We only need to prove the solvability. Since $u$ satisfies \eqref{exist} with the conormal derivative boundary condition on $(0,T)\times\partial \Omega$ {and the zero initial condition on $\{0\}\times \Omega$,} the function $v:=e^{-\kappa t}u$ satisfies
\begin{equation*}
v_t+(-1)^m\l v+\kappa v=e^{-\kappa t}\sum_{|\alpha|\le m}D^\alpha f_\alpha
\end{equation*}
with the corresponding conormal derivative boundary condition and initial condition.
We then reduce the problem to the solvability of $v$. By Lemma \ref{lm8.4}, \eqref{finexist} holds for $v$. Consider the following equation
 \begin{equation*}
 v_t+(-1)^m(s\l v+(1-s)\delta_{ij}\Delta ^{m} v)+\kappa v=e^{-\kappa t}\sum_{|\alpha|\le m}D^\alpha f_\alpha
 \end{equation*}
with the same boundary condition and initial condition, where the parameter $s \in [0,1]$.  It is known that when $s=0$ there is a unique solution in $C^{\frac{a+m}{2m}, a+m}((0,T)\times \Omega)$. Then by the method of continuity  and the a priori estimate \eqref{finexist}, we find a solution when $s=1$.  The theorem is proved.
 \end{proof}

\bibliographystyle{plain}

\begin{thebibliography}{1}

\bibitem{ADN64} \textsc{Agmon S., Douglis A., Nirenberg L.}: Estimates near the boundary for solutions of elliptic partial differential equations satisfying general boundary
conditions, I, \textit{Comm. Pure Appl. Math.}, \textbf{12}  (1959), 623--727; II, ibid., \textbf{17}  (1964), 35--92.

\bibitem{BC93}\textsc{Bramanti M., Cerutti M.}: $W_p^{1, 2}$ solvability for the Cauchy-Dirichlet problem for parabolic equations with VMO coefficients, \textit{Comm. Partial Differential Equations} \textbf{18} (1993), no.9-10, 1735--1763.

\bibitem{BCM96}\textsc{Bramanti M., Cerutti M., Manfredini M.}: $L^p$ estimates for some ultraparabolic operators with discontinuous coefficients, \textit{J. Math. Anal. Appl.} \textbf{200} (1996), no. 2, 332--354.

\bibitem{Cam66}\textsc{Campanato S.}: Equazioni paraboliche del secondo ordine e spazi $\l^{2,\theta}(\Omega ,\delta)$, \textit{Ann. Mat. Pura Appl. (4)} \textbf{73} (1966), 55--102.

\bibitem{FMP91}\textsc{Chiarenza F., Frasca M., Longo P.}: Interior $W^{2,p}$ estimates for nondivergence elliptic equations with discontinuous coefficients, \textit{Ricerche Mat.} \textbf{40} (1991), no. 1, 149--168.

\bibitem{FMP93}\textsc{Chiarenza F., Frasca M., Longo P.}:  $W^{2,p}$-solvability of the Dirichlet problem for nondivergence elliptic equations with VMO coefficients, \textit{Trans. Amer. Math. Soc.} \textbf{336} (1993), no. 2, 841--853.

\bibitem{DK11b}\textsc{Dong H., Kim D.}: $L_p$ solvability of divergence type parabolic and elliptic systems with partially BMO coefficients, \textit{Calc. Var. Partial Differential Equations} \textbf{40} (2011), no. 3-4, 357–-389.

\bibitem{DK11}\textsc{Dong H., Kim D.}: Higher order elliptic and parabolic systems with variably partially BMO coefficients in regular and irregular domains, \textit{J. Funct. Anal.} \textbf{261}, (2011), no. 11, 3279--3327.

\bibitem{DK10}\textsc{Dong H., Kim D.}: On the $L_p$ solvability of higher order parabolic and elliptic systems with BMO coefficients, \textit{Arch. Rational Mech. Anal.} \textbf{199} (2011), no. 3, 889–-941.

\bibitem{DK12}\textsc{Dong H., Kim D.}: The conormal derivative problem for higher order elliptic systems with irregular coefficients, \textit{Contemp. Math.}, \textbf{581} (2012), 69--97.

\bibitem{DZ12}\textsc{Dong H., Zhang H.}: The Schauder estimates for higher order parabolic systems with time irregular coefficients, \textit{preprint}, arXiv:1307.5013.


\bibitem{F96}\textsc{Di Fazio G.}: $L^p$ estimates for divergence form elliptic equations with discontinuous coefficients, \textit{Boll. Un. Mat. Ital. A (7)} \textbf{10} (1996), no. 2, 409--420.

\bibitem{FA}\textsc{Friedman A.}: ``Partial differential equation of parabolic type", Prentice-Hall, Englewood cliffs, N.J, 2008.

\bibitem{Gia93}\textsc{Giaquinta M.}: ``Introduction to regularity theory for nonlinear elliptic systems '', Lectures in Mathematics ETH Z\"urich. Birkh\"auser Verlag, Basel, 1993.

\bibitem{Knerr} \textsc{Knerr B. F.}:
Parabolic interior Schauder estimates by the maximum principle,
\textit{Arch. Rational Mech. Anal.} \textbf{75} (1980/81), no. 1, 51--58.


\bibitem{L83}\textsc{Lieberman G.}: The conormal derivative problem for elliptic equations of variational type, \textit{J. Differential Equations} \textbf{49} (1983), no. 2, 218--257.


\bibitem{L87}\textsc{Lieberman G.}: H\"older continuity of the gradient of solutions of uniformly parabolic equations with conormal boundary conditions, \textit{Ann. Mat. Pura Appl. (4)} \textbf{148} (1987), 77--99.

\bibitem{L92}\textsc{Lieberman G.}: Intermediate Schauder theory for second order parabolic equations IV: time irregularity and regularity, \textit{Differential and Integral Equations} \textbf{5} (1992), no. 6,  1219--1236.

\bibitem{L921}\textsc{Lieberman G.}: The conormal derivative problem for equations of variational type in nonsmooth domains, \textit{Trans. Amer. Math. Soc.} \textbf{330} (1992), no. 1, 41--67.

\bibitem{L96}\textsc{Lieberman G.}: ``Second order parabolic differential equations", World Scientific Publishing Co., Inc., River Edge, NJ, 1996.


\bibitem{L13}\textsc{Lieberman G.}: ``Oblique derivative problem for elliptic equations", World Scientific Publishing Co., Inc, Hackensack, NJ, 2013.

\bibitem{Lorenzi} \textsc{Lorenzi L.}:
Optimal Schauder estimates for parabolic problems with data measurable with respect to time,
\textit{SIAM J. Math. Anal.} \textbf{32} (2000), no. 3, 588--615.

\bibitem{Kry96}\textsc{Krylov N.V.}: ``Lectures on elliptic and parabolic equations in H\"{o}lder spaces". American Mathematical Society, Providence, RI, 1996.

\bibitem{Kry04}\textsc{Krylov N.V.}: On weak uniqueness for some diffusions with discontinuous coefficients, \textit{Stochastic Process. Appl.}, \textbf{113} (2004), no. 1, 37--64.


\bibitem{Kry07}\textsc{Krylov N.V.}: Parabolic and elliptic equations with VMO coefficients, \textit{Comm. Partial Differential Equations} \textbf{32} (2007), no. 1-3, 453--475.

\bibitem{Kry071}\textsc{Krylov N.V.}: Parabolic equations with VMO coefficients in Sobolev spaces with mixed norms, \textit{J. Funct. Anal.} \textbf{250} (2007), no. 2, 521--558.

\bibitem{Kry08}\textsc{Krylov N.V.}:  ``Lectures on elliptic and parabolic equations in Sobolev spaces", American Mathematical Society, 2008.

\bibitem{KrPr} \textsc{Krylov N. V., Priola E.}:
Elliptic and parabolic second-order PDEs with growing coefficients,
\textit{Comm. Partial Differential Equations.} \textbf{35} (2010), no. 1, 1--22.

\bibitem{MPS}\textsc{Maugeri A., Palagachev D., Softova L.}: Elliptic and Parabolic Equations with Discontinuous Coefficients, \textit{Mathematical Research}, \textbf{109}. Wiley-VCH, Berlin, 2000.


\bibitem{PS} \textsc{Palagachev D., Softova L.}:
A priori estimates and precise regularity for parabolic systems with discontinuous data, \textit{Discrete Contin. Dyn. Syst.} \textbf{13} (2005), no. 3, 721--742.

\bibitem{SB12}\textsc{Boccia S.}: Schauder estimates for solutions of higher-order parabolic systems, \textit{Methods Appl. Anal.} \textbf{20} (2013), no. 1, 47--67.

\bibitem{Sh96}\textsc{Schlag W.}: Schauder and $L^p$ estimates for parabolic system via Campanato spaces, \textit{Comm. Partial Differential Equations} \textbf{21} (1996), 1141--1175.

\end{thebibliography}

\def\cprime{$'$}\def\cprime{$'$} \def\cprime{$'$} \def\cprime{$'$}
  \def\cprime{$'$} \def\cprime{$'$}

\end{document}